\documentclass[a4paper, 11pt]{article} 

\usepackage{times}
\usepackage{graphicx}
\usepackage{amssymb,amsmath}
\usepackage{upref}
\usepackage[sf,SF]{subfigure} 
\usepackage{caption}
\usepackage{amsthm,euscript}
\usepackage[a4paper, portrait, margin=1in]{geometry}
\usepackage{xcolor}
\usepackage[bookmarks=true,hidelinks]{hyperref}
\usepackage{units}
\usepackage{bbm}
\usepackage{enumitem}
\usepackage{caption}
\usepackage[title]{appendix}

\newcommand{\R}{\mathbb{R}}

\newcommand{\Z}{\mathbb{Z}}
\renewcommand{\leq}{\leqslant}
\renewcommand{\geq}{\geqslant}
\renewcommand{\phi}{\varphi}
\newcommand{\hf}{{\unitfrac{1}{2}}}
\renewcommand{\epsilon}{\varepsilon}
\newcommand{\bx}{\mathbf{x}}
\newcommand{\by}{\mathbf{y}}
\newcommand{\br}{\mathbf{r}}
\newcommand{\bk}{\mathbf{k}}

\newcommand{\fstat}{f_\infty}

\DeclareMathOperator\erf{erf}
\newcommand{\eps}{{\varepsilon}}

\newtheorem{theorem}{Theorem}
\newtheorem{proposition}[theorem]{Proposition}

\newtheorem{remark}[theorem]{Remark}

\newtheorem{lemma}[theorem]{Lemma}

\numberwithin{equation}{section}
\numberwithin{theorem}{section}

\makeatletter
\def\@cite#1#2{\textup{[{#1\if@tempswa , #2\fi}]}}
\makeatother

\title{Noise-driven bifurcations in a neural field system \\modelling networks of grid cells}

\author{Jos\'e A. Carrillo\thanks{Mathematical Institute, University of Oxford, Oxford OX2 6GG, UK (carrillo@maths.ox.ac.uk)} \and Helge Holden\thanks{Department of Mathematical Sciences, NTNU Norwegian University of Science and Technology, NO-7491 Trondheim, Norway (helge.holden@ntnu.no)} \and Susanne Solem\thanks{Department of Mathematics, Norwegian University of Life Sciences, NO-1433 \AA s, Norway (susanne.solem@nmbu.no)}}

\begin{document}

\maketitle

\begin{abstract}
The activity generated by an ensemble of neurons is affected by various noise sources. It is a well-recognised challenge to understand the effects of noise on the stability of such networks. We demonstrate that the patterns of activity generated by networks of grid cells emerge from the instability of homogeneous activity for small levels of noise. This is carried out by analysing the robustness of network activity patterns with respect to noise in an upscaled noisy grid cell model in the form of a system of partial differential equations. Inhomogeneous network patterns are numerically understood as branches bifurcating from unstable homogeneous states for small noise levels. We show that there is a phase transition occurring as the level of noise decreases. Our numerical study also indicates the presence of hysteresis phenomena close to the precise critical noise value.
\end{abstract}

{\bf Keywords: }{grid cells, noise-driven bifurcations, neural field models, partial differential equations.

} 

\section{Introduction}
By now it is well established that grid cells, and the characteristic hexagonal firing patterns they create in physical space, play an important role in the navigational system of mammalian brains \cite{McNaughtonMoser}. Since grid cells were discovered in 2005 \cite{gridcells}, there has been extensive activity in order to understand their precise behaviour, see  \cite{tenyears,McNaughtonMoser} and the references therein. The main challenges ten years after the discovery of grid cells, such as how grid cells are organised and how they are connected to other cell types in the brain, were highlighted in \cite{tenyears}. In particular, as the brain is inherently noisy \cite{rolls2010noisy}, the lack of understanding of the effect of noise on grid cells was emphasised as a challenge. The recent results in \cite{Gardner2022} has provided insight into the organisation of grid cells by showing that the activity of the network (called a module in \cite{Gardner2022}) is arranged on a torus. The question regarding the effects of noise on grid cells, however, remains open. 

In accordance with previous experimental studies and general belief in the field, the results in \cite{Gardner2022} provided further evidence in favour of describing the grid cell network by continuous attractor network dynamics through a system of neural field models \cite{Ermentrout2010}. The first attractor network models for grid cells were presented in \cite{McNaughtonetal, burakfiete, coueyetal}, which were based on the classical papers \cite{WC1, WC2, Amari1977}, see also \cite{whiskerbarrel}. In \cite{burakfiete, coueyetal}, the grid cells are assumed to have orientation preferences in four different directions. The hexagonal grid cell patterns are then generated by a system of $4N^2$ neural field ordinary differential equations
\begin{align}\label{eq:motherODE}
\tau_{i}^\beta \frac{d s_i^\beta}{dt} + s_i^\beta = \Phi\left( \sum_{\beta'} \sum_j W_{ij}^{\beta'} s_j^{\beta'} + B_i^\beta (t) \right),
\end{align}
with $\beta = 1, \dots, 4$. Here $s_i^\beta \geq 0$ represents the activity level of neuron $i$ with orientation preference $\beta$, and $\tau_{i}^\beta$ is its relaxation time. The right-hand side of \eqref{eq:motherODE} represents the firing rate of the neuron, see \cite{bressloff2012}. The function $\Phi$ is a given activation function, often of the form of a ReLU or sigmoid function. The firing rate of neuron $i$ depends on an external input $B_i^\beta(t)$ and the response of the network. It is assumed that the neurons are arranged on a square, which we will denote $\Omega$ and call the neural sheet, according to the strength of their pairwise connection. The position of neuron $i$ is denoted by $\bx_i$. The strength of the connectivity between neuron $i$ of type $\beta$ and $j$ of type $\beta'$ is $W_{ij}^{\beta'} = W(\bx_i-\bx_j-\br^{\beta'})$, where $W(\bx$), $\bx \in \Omega$, is assumed to be even in each coordinate, in $\Omega$. The connectivity is shifted in the direction of the orientation preference of neuron $j$ of type $\beta'$ with $\br^{\beta'}$ which is given by shifts of equal length in the four cardinal directions: north, south, east and west. It has been commonly considered, and, as mentioned, recently shown in \cite{Gardner2022}, that the network of neurons creates a torus connectivity. This is realised in the model by assuming that $W$ is extended periodically outside $\Omega$.

 By modelling the movement of a rat traversing physical space through the input $B_i^\beta(t)$, \eqref{eq:motherODE} can recreate the hexagonal patterns in physical space produced by the firing of a single grid cell as observed in experiments \cite{coueyetal}. In this model, the patterns in physical space are a consequence of the patterns generated on the neural sheet $\Omega$. However, \eqref{eq:motherODE} being a deterministic model, it does not offer much insight into the effects of noise. 
 
Works on understanding noisy neural fields have in general been lacking \cite[Sec. 6]{bressloff2012} until recently \cite{TouboulPhysD,KE13,K14,MB, TouboulSIAM, B19, BAC19}. In \cite{BurakFieteNoise} fundamental limits on how information dissipates in networks of noisy neurons were derived. The author in \cite{K14} presents a study of two coupled noisy neural field models with a focus on the consequences of the coupling on the neural activity waves, while \cite{KE13,MB} investigate the effect of noise on stationary bumps in one-dimensional spatially extended networks.

Taking a different perspective than \cite{BurakFieteNoise, KE13, K14, MB}, by adding noise to the common model of a grid cell network \eqref{eq:motherODE}, the main goal of this work is to analyse the robustness of the hexagonal patterns in the activity level \cite{EC1979} observed in \eqref{eq:motherODE} with respect to noise strength. We show that the stationary spatial patterns of the activity level emerge from the instability of homogeneous brain activity as the noise level diminishes. By upscaling the model \eqref{eq:motherODE} with noise to a system of Fokker--Planck-like partial differential equations, our analysis gives an estimate on the noise strength above which there is no coherent activity pattern. We also numerically explore the different branches of inhomogeneous stationary patterns bifurcating from the homogeneous state depending on the noise for several activation functions $\Phi$ indicating the presence of hysteresis phenomena. 
 
Instabilities of homogeneous steady states of noisy neural fields were also investigated in \cite{BAC19} utilising a partial differential equation (PDE) description. However, the PDEs were of a very different form than the ones presented in this manuscript. A Fokker--Planck-like system describing a network of noisy neurons can be found in \cite{B19}, where neural variability in a coupled ring network was studied.

It is classical to analyse the behavioural change of neural fields without noise in the form of ordinary differential equations (ODEs) by standard bifurcation analysis \cite{Murray03,bressloff2012,VCF15,KP17,SA20}. Finding noise-driven bifurcations is more challenging, and one has to rely on other technical tools unless the coupling of the network has a particular structure where closed ODEs for the mean and the variance are available \cite{TouboulSIAM,TouboulPhysD}. The system \eqref{eq:motherODE} with noise has, in addition, a nonlinear coupling, and the activity levels must remain nonnegative due to their physical interpretation, leading to technical additional constraints on the stochastic processes involved. In the following we deal with these challenges by analysing a system of Fokker--Planck-like partial differential equations with boundary conditions describing the space-time evolution of the law of the stochastic processes with respect to the noise level. 

\section{The PDE system: derivation, main goal, and numerical experiment}

For the sake of the reader, we start by discussing the simplest classical case of no spatial connectivity \cite{Hopfield84,bressloff2012} and the references therein. Let us consider the classical neural field stochastic dynamical system  for a network of $M$ coupled neurons given by
\begin{align}\label{eq:SDEneuralfield}
\tau d s_{k} + s_{k} dt = &\,\Phi \Bigg(\frac{W_{0}}{M}\sum_{k'} s_{k'} + B(t) \Bigg) dt + \sqrt{2\sigma} d\mathcal{W}_{k}.
 \end{align}
Here, the neurons are considered indistinguishable and all-to-all coupled with equal strengths given by $W_0 \in \R$ whose sign depends on the type of neurons considered: inhibitory or excitatory. We also consider that the relaxation time for all neurons is the same and equal to $\tau$. $B(t)$ is the external input for this neural network and $\sigma>0$ is the strength of the noise $\mathcal{W}_k$. We have considered independent Brownian motion for each neuron in the network. Classical stochastic analysis implies that we can derive a Fokker--Planck equation for the evolution of the probability density of neurons with activity level $s$ at time $t$ in the large population limit $M\to\infty$, i.e., the law of the limiting stochastic process follows the PDE
\begin{equation}\label{eq:PDEhom-noBC}
\tau \frac{\partial f}{\partial t} =
\frac{\partial}{\partial s}\left(
\left[s-\Phi\big(W_0\langle f \rangle + B(t) \big) \right] f
\right) + \sigma \frac{\partial^2 f}{\partial s^2},
\end{equation}
where $f=f(t,s)$ denotes the probability to observe the activity $s$ at time $t$, and $\langle f \rangle$ denotes the mean value of the activity level $s$
\begin{align*}
\langle f \rangle = \int_{0}^\infty s f (s)\, ds.
\end{align*}
Notice that the noise can drive the activity level to be negative in \eqref{eq:SDEneuralfield}, which is clearly not desirable from the modelling viewpoint. In order to avoid this, it is common practise to consider the Fokker--Planck equation \eqref{eq:PDEhom-noBC} on $s\in [0,\infty)$ with no-flux boundary conditions
\begin{equation}\label{eq:PDEhomBC}
\Big( \Phi\big(W_0\langle f \rangle + B(t) \big) f - \sigma \frac{\partial}{\partial s} f \Big) \bigg|_{s=0}   = 0.
\end{equation}
This ensures that particles cannot escape from non-negative values of the activity level variable $s$ at the PDE level while keeping an evolution of a probability density, see \cite{CCM11,CCM13} for instance. 

\begin{remark}[Microscopic Model]
\label{rem:microscopic}
Reflective boundary conditions for stochastic processes have been incorporated at the stochastic differential equation (SDE) level in order to avoid particles to escape a fixed domain \cite{Sznitman1984,LS84,FI15}. One can produce a microscopic stochastic process by adding an additional process counting when particles touch the boundary of the domain. The law of the rigorous mean-field limit, $M\to\infty$, of the following system 
\begin{align}\label{eq:SDE1hom}
\tau d s_{k} + s_{k} dt = &\,\Phi \Bigg(\frac{W_{0}}{M}\sum_{k'} s_{k'} + B(t) \Bigg) dt+ \sqrt{2\sigma}\, d\mathcal{W}_{k} - dl_{k}, \\
 l_{k}(t) & = -|l_{k}|(t), \quad |l_{k}| (t) = \int_0^t 1_{\{s_{k}(\zeta) = 0\}}d|l_{k}|(\zeta),
\end{align}
$k =1,\dots, M$, follows the evolution of \eqref{eq:PDEhom-noBC}--\eqref{eq:PDEhomBC} under suitable smoothness assumptions on $\Phi$, see \cite{LS84,Sznitman1991}. 
\end{remark}

The next step in the modelling is to reinterpret $M$ as the number of neurons in each of the cortical columns of a neural sheet of $N$ columns. 
Given space points $x_1,\dots,x_N$ in the region $\Omega$ of the neural cortex, the interaction among $NM$ neurons stacked in $N$ columns at locations $x_i$ with $M$ neurons each, where $s_{ik}^\beta$ represents the activity level with orientation $\beta$ of the $k^{th}$ neuron at location $x_i$ is given for $i=1,\dots,N$ and $k=1,\dots,M$ by
\begin{subequations}
\label{concrete model with reflection term}
\begin{align}
        \tau ds^\beta_{ik} + s_{ik}^\beta dt =&\, \Phi\left(\frac{1}{4N M}\sum_{\beta'=1}^4\sum_{j=1}^{N}\sum_{m=1}^{M} W^{\beta'}(x_i-x_j)s_{jm}^{\beta'} + B^\beta(t)\right) dt + \sqrt{2\sigma} d\mathcal{W}_{ik}^\beta-d\ell^\beta_{ik},
        \label{concrete model particle system} \\
        \ell_{ik}^{\beta}(t)= &\, -\big|\ell_{ik}^{\beta}\big|(t),\quad  \big|\ell_{ik}^{\beta}\big|(t)=\int_0^t1_{\{s_{ik}^\beta(r)=0\}}d\big|\ell_{ik}^{\beta}\big|(r)\quad\text{for }\beta=1,2,3,4.
    \label{reflection term concrete model particle system}
  \end{align}
  \end{subequations}
Here, we consider the same periodic setting, imposed through the periodicity of the interactions $W^{\beta}$ for $\beta=1,\dots,4$, as in \cite{burakfiete,coueyetal}. Moreover, the neurons are inhibitory \cite{coueyetal} and the activity in the network is modulated by a time dependent external input as in \eqref{eq:motherODE}. We are dealing with a population network of neurons structured by their orientation preference $\beta=1,\dots,4$ corresponding to the four cardinal points (north, west, south, east). The network population includes the localised in space cross inhibition of neurons with different orientations modulated by the shape function $W^\beta$, where $W^\beta(\bx)=W(\bx-{\bf r}^\beta)$. Following the approach outlined above in the case of one population, we can formally write a Fokker--Planck type equation, in the limit $N,  M \to \infty$, for the evolution of the probability density $f^\beta(t,\bx,s)$ of finding neurons of type $\beta$ at position $\bx$ on the neural sheet $\Omega$ with activity level $s\geq 0$ at time $t\geq 0$. We refer to \cite{CTRM06} for a similar approach in conductance-voltage models. The system of equations reads
\begin{align}\label{eq:PDE}
\tau \frac{\partial f^\beta}{\partial t} =
-\frac{\partial}{\partial s}\Bigg(
\Big[\Phi^\beta(\bx) -s\Big] f^\beta
\Bigg) + \sigma \frac{\partial^2 f^\beta}{\partial s^2},
\end{align}
where $\Phi^\beta(\bx)$ is given by
\begin{align*}
\Phi \left(\frac{1}{4}\sum_{\beta'} \int_\Omega W^{\beta'}(\bx-\by) \langle f^{\beta'} \rangle (t,\by)d\by + B^\beta (t) \right),
\end{align*}
with
\begin{align*}
\langle f^{\beta} \rangle (t,\bx) &= \int_{0}^\infty s f^{\beta} (t,\bx,s)\, ds, \quad \beta=1,\dots,4,
\end{align*}
periodic boundary conditions in $\bx$, and the no-flux boundary conditions at $s=0$ given by 
\begin{align}\label{eq:PDEbc}
\Bigg( \Phi^\beta(\bx) f^\beta - \sigma \frac{\partial}{\partial s} f^\beta \Bigg) \Bigg|_{s=0}\!\!\!\!\!\!\!\!\! = 0, \quad \beta=1,\dots,4,
\end{align}
for each position $\bx$ in the square sheet $\Omega$. To realise the torus connectivity, we assume that $W$ is periodic with respect to $\Omega$ and even in each coordinate on $\Omega$. The function $\Phi$ is typically a smooth approximation of the ReLU activation function $(x)^+= \max\{0,x\}$ or a sigmoid function. The initial probability density of the system \eqref{eq:PDE} is denoted by $f_0^\beta$.

\begin{remark}
The system of Fokker--Planck equations \eqref{eq:PDE}--\eqref{eq:PDEbc} can be rigorously derived from the microscopic stochastic processes \eqref{concrete model with reflection term} under suitable assumptions. The rigorous proof of this mean-field limit for the spatially extended system \eqref{eq:PDE}--\eqref{eq:PDEbc} has recently been obtained in \cite{CCS21} by a generalisation of the coupling method of Sznitman \cite{Sznitman1991}. This rigorous passage to the limit is a very interesting area of mathematical research on its own with a multitude of different models and limiting systems derived under different assumptions on the ingredients of the network. For instance, we refer to the works \cite{MS02, FTC09,FI15,TouboulPhysD,TouboulSIAM,CT18} in which the authors deal with spatially extended systems of neural networks modelled by their voltage with random connectivity interactions using large deviation principles \cite{AG95,G97}.
\end{remark}

To summarise, the main goal of this work is to focus on the biological information carried by the system of PDEs \eqref{eq:PDE}--\eqref{eq:PDEbc}. More precisely, we study how noise affects the dynamics of \eqref{eq:PDE}--\eqref{eq:PDEbc} under the following assumptions: 
\begin{enumerate}[label=(A{\arabic*})]
\setlength\itemsep{0.2em}
    \item \label{assumptionW1} The grid cells are arranged on a torus, realised by setting the neural sheet $\Omega=[-0.5,0.5]^2$ and extending $W$ periodically outside $\Omega$.
    \item  \label{assumptionW2} The inhibitory \cite{coueyetal} connectivity function $W \leq 0$ is at least in $L^2(\Omega)$, and is an even function in each coordinate on $\Omega$. Furthermore, we define $W_0 = \int_\Omega W(\bx)d\bx$. In the numerical experiments $W$ satisfies $W(\bx)=W(|\bx|)$ in addition.
    \item \label{assumptionphi} The modulation function $\Phi$ is in $C^1$ (unless otherwise stated).
    \item  \label{assumptionrbeta} There are four orientation preferences, $\beta=$ $1$ (north), $2$ (west), $3$ (south), and $4$ (east), where the shifts are of equal size $z$ in each direction, i.e., $\br^\beta = z{\bf e}_\beta$, where ${\bf e}_\beta$ is the unit vector in direction $\beta$. 
\end{enumerate}

It is well-known that grid cell firing is strongly connected to mammals' navigation, but unknown exactly how the grid cell network communicates with other networks in the brain. We will therefore simply assume in the numerical experiments that the external input $B^\beta(t)$ in \eqref{eq:PDE}--\eqref{eq:PDEbc} depends on the velocity at time $t$, $v(t)$, of a moving animal in the following manner (see \cite{burakfiete,coueyetal}):
\begin{align}\label{eq:ext}
    B^\beta(t) = B + \alpha v(t) \cos (\theta (t) - \theta^\beta),
\end{align}
where $B>0$ is a constant external excitatory input, assumed to be the same for different $\beta$, $\alpha$ the velocity modulation, $\theta (t)$ the orientation of the animal at time $t$ according to the reference frame, and $\theta^\beta$ the orientation preference of the neurons of type $\beta$ ($\theta^1=\frac{\pi}{2}, \theta^2=\pi, \theta^3=\frac{3}{2}\pi, \theta^4=2\pi$). This particular form of the input, together with the right set of parameters in \eqref{eq:motherODE}, has been shown to enable single cells of the ODE system \eqref{eq:motherODE} to create hexagonal firing patterns in physical space, see \cite{burakfiete, coueyetal}.

\subsection{Numerical reproduction of the hexagonal patterns}
We numerically demonstrate that the PDE system \eqref{eq:PDE}--\eqref{eq:PDEbc} with \eqref{eq:ext} is able to reproduce the characteristic single-cell hexagonal firing pattern as discovered in \cite{gridcells} for rats and see how this pattern depends on the noise strength $\sigma>0$. For this, we use a numerical scheme that has been extensively utilised for Fokker--Planck like equations \cite{CCH}. For more details on the numerical approach and its validation, see Appendix \ref{app:numerics}. Before connecting the grid cell system \eqref{eq:PDE}--\eqref{eq:PDEbc} with the movement of a rat, we initialise the activity on the neuronal sheet $\Omega$ by running the simulation with $\alpha=0$ until $f^\beta$ has numerically stabilised into stationary patterns equal to the ones in the top and middle rows of Fig.~\ref{fig:single}, modulo translations.

\begin{figure}
\centering
\subfigure[$\sigma = 0.001$]
{
\includegraphics[trim={0cm 2.5cm 0cm 2.5cm},clip, width=0.22\textwidth]{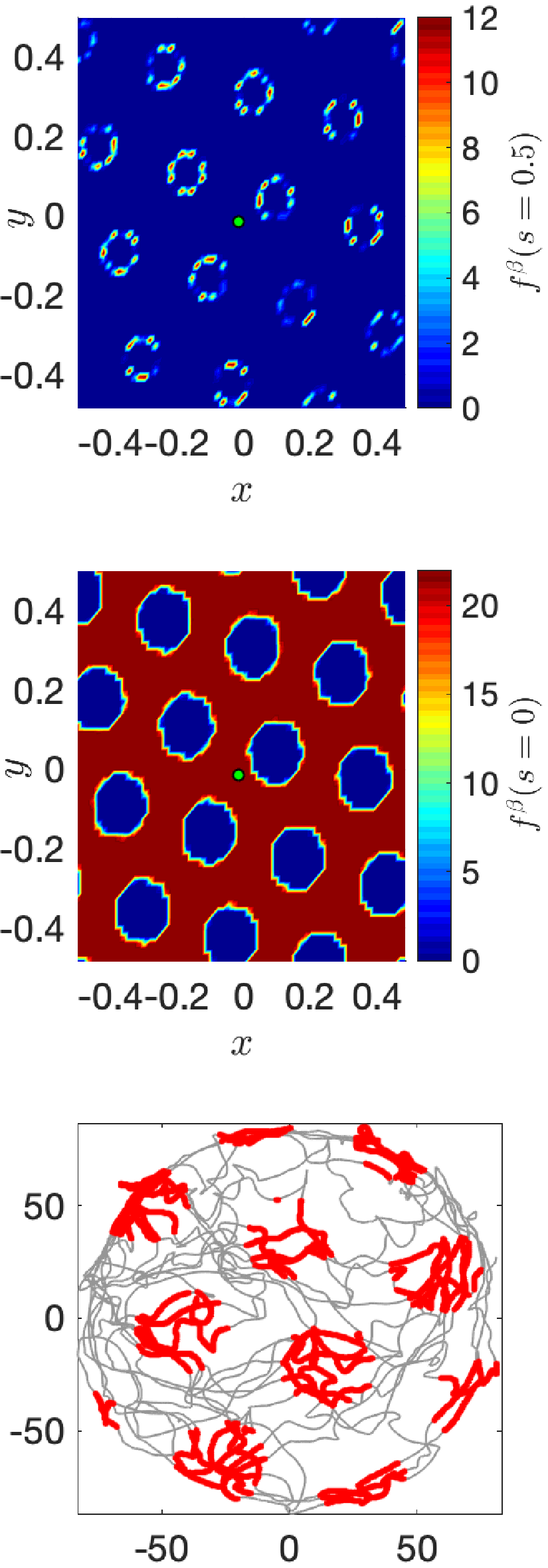}}
\subfigure[$\sigma = 0.005$.]
{
\includegraphics[trim={0cm 2.5cm 0cm 2.5cm},clip,width=0.22\textwidth]{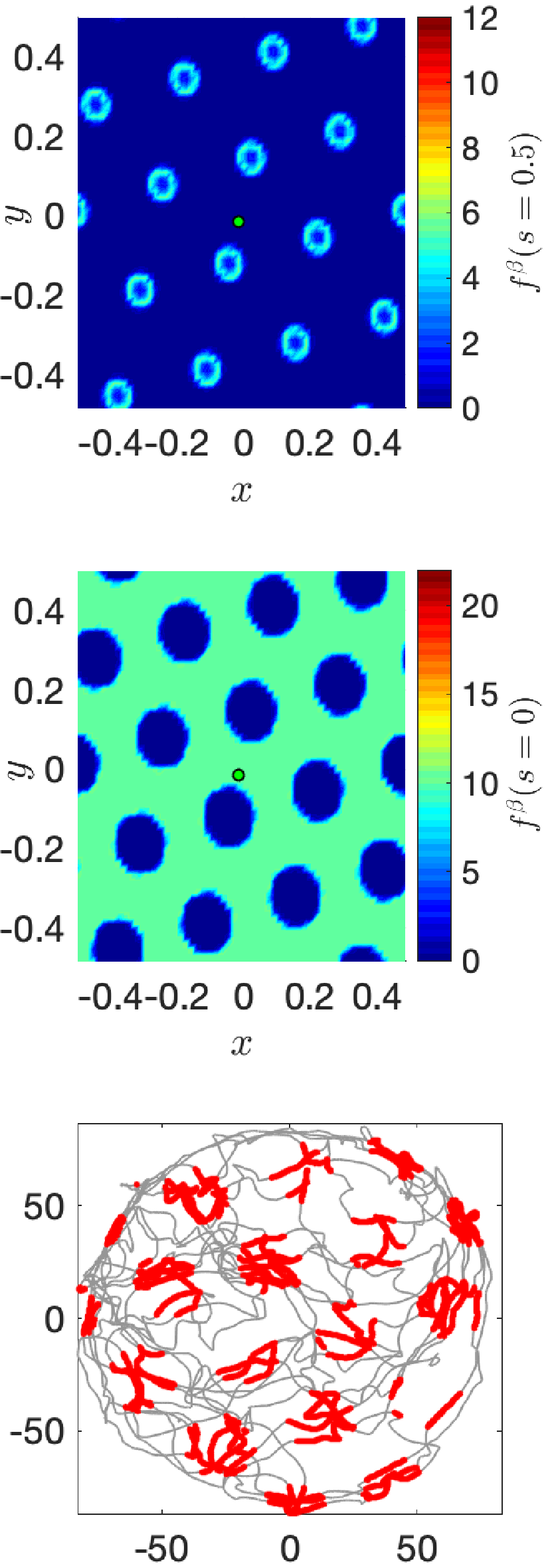}}
\subfigure[$\sigma = 0.015$.]
{
\includegraphics[trim={0cm 2.5cm 0cm 2.5cm},clip,width=0.22\textwidth]{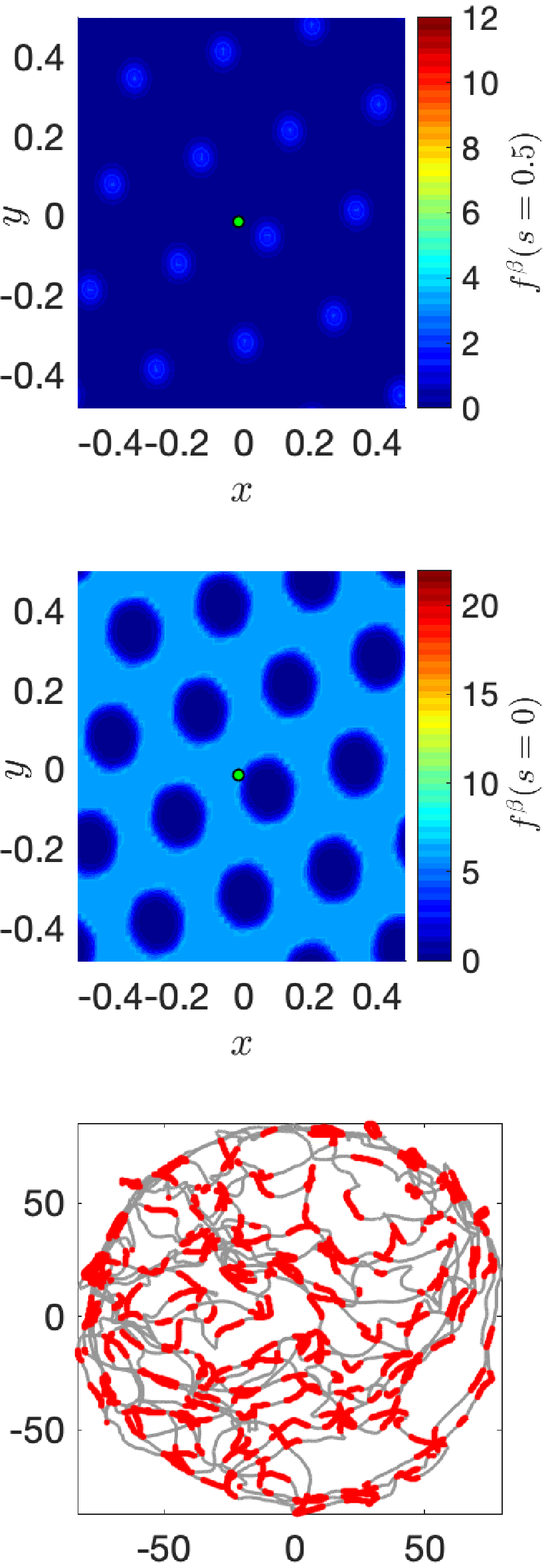}}
\subfigure[$\sigma = 0.02$]
{\includegraphics[trim={0cm 2.5cm 0cm 2.5cm},clip,width=0.22\textwidth]{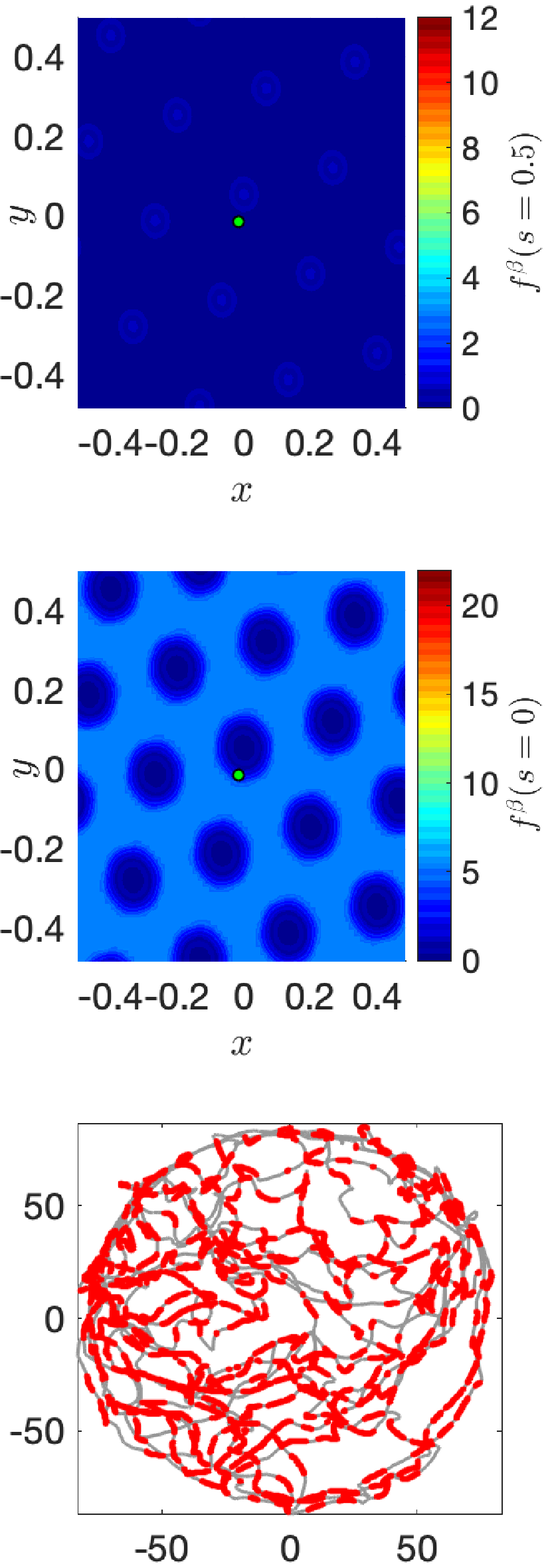}}
\caption{Network patterns and single-cell responses after $t=5$ minutes for increasing noise strength $\sigma$ (left to right $\sigma = 0.001, 0.005, 0.015, 0.02$) with the modulation function used in \cite{coueyetal}: $\Phi(x)=(x)^+$ and \eqref{eq:ext} with $I=3$ and $\alpha=0.3$. Top row: the probability density $f^\beta(x,y,s=0.5)$ for one $\beta$, $(x,y) \in \Omega$, middle row: $f^\beta(x,y,s=0)$, and bottom row: single-cell firing pattern in a circular enclosure traversed by a rat (units in cm) created by the cell at position $(0,0)$ on the neural sheet $\Omega$ (green dot in top and middle plots).}
\label{fig:single}
\end{figure} 

Then we connect with the rats movement by setting $\alpha=0.3$ in \eqref{eq:ext} as in \cite{coueyetal}. The velocity $v(t)$ and orientation $\theta(t)$ used in the numerical experiment are calculated using timestamped position data from the physical experiments in \cite{gridcells} where rats moved around in a circular enclosure with a radius of $80$ cm. The shape of the enclosure can be seen in the plots in the bottom row of Fig.~\ref{fig:single}. The path of the rat after moving around for $t=5$ minutes, generated with the position data, is visualised in grey.

The red coloured areas in each plot in the bottom row in Fig.~\ref{fig:single} make up the firing field of a grid cell in the network. The firing field consists of smaller red circular areas, which again are made up by even smaller red dots. Each dot marks a firing of the grid cell placed at position $(0,0)$ on the neuronal sheet $\Omega$ as the velocity and orientation data is fed into the numerical method of \eqref{eq:PDE}--\eqref{eq:PDEbc} through \eqref{eq:ext}. For simplicity, we have assumed in the numerical experiment that a neuron at a particular position $\bx \in \Omega$ fires as soon as the firing rate $\Phi$ in \eqref{eq:PDE}--\eqref{eq:PDEbc} satisfies $\Phi(\bx) > 0$.

Now, from left to right in the bottom row of Fig.~\ref{fig:single}, we see the pattern of the firing fields---the patterns created by the red dots over the path of the rat in the enclosure in physical space---for the grid cell at $(0,0)$ for increasing noise strength.

The top and middle rows in the figure display snapshots of the probability density $f^\beta$ at $s=0.5$ and $s=0$, respectively. As $t$ increases, these patterns are translated in accordance with the movement of the rat.

As can be observed in the bottom row, the red fields generated with the PDE system \eqref{eq:PDE}--\eqref{eq:PDEbc} form hexagonal patterns similar to the ones observed in physical experiments. However, the distance between the activity bumps and the area they cover decreases as the noise strength increases. The second main observation is that by increasing the noise the firing becomes less and less localised. The numerical experiments support the existence of a critical value of the noise, $\sigma_c>0$, at which a single cell could fire no matter where the rat is on its path.

The question we will address in the following is how stable the patterns observed in Figure \ref{fig:single} are with respect to the noise strength $\sigma$. 

\section{Stability of the neural field system}\label{sec:stability}

In this section, we start by studying the spatially homogeneous solutions. We would like to understand the pattern formation in the neural field system \eqref{eq:PDE}--\eqref{eq:PDEbc} as a byproduct of the instability of these homogeneous solutions. We assume that solutions to \eqref{eq:PDE} are sufficiently smooth and decay fast enough as $s \to \infty$. We further assume from now on that $B^\beta=B$ is constant, i.e., $\alpha=0$ in \eqref{eq:ext}, in order to study the emergence of stationary network patterns of the system. Note that setting $B^\beta$ to different constant values depending on $\beta$ could yield non-stationary network patterns: a rat running with constant speed in one direction would give $v(t) =const>0$ and $\theta (t) =const$ in \eqref{eq:ext}, which results in different constant values of $B^\beta$. This would, with the right set of parameter values, consequently translate the network patterns in time. To avoid this technicality, we let $B$ be identical for the four different direction preferences.

Homogeneous solutions $f^\beta (t,\bx,s) = f(t,s)$ for $\beta = 1, \dots, 4$ to \eqref{eq:PDE} satisfy
\begin{equation}\label{eq:PDEhom}
\tau \frac{\partial f}{\partial t} =
\frac{\partial}{\partial s}\left(
\left[s-\Phi\big(W_0\langle f \rangle + B \big)\right] f
\right) + \sigma \frac{\partial^2 f}{\partial s^2},
\end{equation}
with no-flux boundary conditions \eqref{eq:PDEbc} at $s=0$, i.e., \begin{align}\label{eq:PDEbc-hom}
\Bigg( \Phi^\beta\big(W_0 \langle f \rangle + B \big) f^\beta - \sigma \frac{\partial}{\partial s} f^\beta \Bigg) \Bigg|_{s=0}\!\!\!\!\!\!\!\!\! = 0, \quad \beta=1,\dots,4,
\end{align} and
\begin{align*}
W_0 = \int_\Omega W(\bx) d\bx.
\end{align*}
In order to find stationary spatially homogeneous states $\fstat$ we first assume that their mean $\langle \fstat  \rangle$ is given. Denote by $\Phi_0 = \Phi\big(W_0 \langle \fstat  \rangle + B\big)$ the corresponding firing rate for simplicity. Thus,  by integrating \eqref{eq:PDEhom} and using the boundary condition \eqref{eq:PDEbc-hom}, the stationary spatially homogeneous states $\fstat (s)$ are given by 
\begin{align}\label{eq:stat}
\fstat (s) = \frac{1}{Z}  \exp \left( - \frac{(s-\Phi_0)^2}{2\sigma}\right),
\end{align}
with the mass normalisation factor $Z$ such that $\int_0^\infty \fstat (s) ds=1$, i.e., 
\begin{align}\label{eq:stat2}
Z = \sqrt{\frac{\pi \sigma}{2}}\left( 1+\erf\left(\frac{\Phi_0}{\sqrt{2\sigma}}\right)\right),
\end{align}
where the error function has been defined as
\begin{align*}
    \erf(x) = \frac{2}{\sqrt{\pi}} \int_0^x \exp(-y^2)\,dy.
\end{align*}
However, note that \eqref{eq:stat} is an implicit equation as $\Phi_0$ depends on the mean $\langle \fstat  \rangle$.
To show the existence of stationary solutions, we need to solve the consistency equation for the mean 
$\langle \fstat  \rangle$ given by 
\begin{align}\label{eq:spathomsol}
\langle \fstat  \rangle = \Phi_0 + \sigma \frac{1}{Z} \exp \left( -\frac{\Phi_0^2}{2\sigma}\right) \,.
\end{align}
We prove next that the stationary state exists and is unique by leveraging on \eqref{eq:spathomsol} under suitable conditions on the firing rate function $\Phi$.

\begin{proposition}\label{prop:exisuniquemean}
Let $ \sigma > 0$, $W_0 \leq 0$, and $0\leq \Phi(x) \leq \Phi(B)$ for any $x \leq B$. Assume \ref{assumptionphi}, and that $\Phi$ satisfies $\Phi'(W_0 m +B) > \frac{1}{W_0}$
for all $m\geq 0$, 
then \eqref{eq:PDEhom} has a unique stationary solution $\fstat$ defined by 
\eqref{eq:stat}--\eqref{eq:spathomsol}.
\end{proposition}
\begin{proof}
Define for $m\geq 0$ and $\sigma>0$ the function
\begin{align*}
G(m, \sigma) = \Phi(W_0m+B) + \sigma \frac{1}{Z} \exp \left( -\frac{\Phi^2(W_0m+B)}{2\sigma}\right) - m.
\end{align*}
First, notice that $G(m, \sigma)$ satisfies $G(0,\sigma)>0$ and 
$$
G(m,\sigma) \leq \Phi(W_0m+B) + \sqrt{\frac{\sigma}{2 \pi}}- m \leq  \Phi(B) + \sqrt{\frac{\sigma}{2 \pi}} - m<0
$$ 
for $ m > \Phi (B) + \sqrt{\frac{\sigma}{2 \pi}}$. We now compute
$$
\frac{\partial G}{\partial m}(m,\sigma) = -1 +\Phi'(W_0m+B)W_0\, g\left(\frac{\Phi(W_0m+B)}{\sqrt{2\sigma}}\right)
$$
with 
$$
g(\eta)= \left( 1-\frac{2}{\sqrt{\pi}}\frac{\exp (-\eta^2)}{1+\erf (\eta)}\left[\frac{1}{\sqrt{\pi}}\frac{\exp (-\eta^2)}{1+\erf (\eta)}+\eta\right]\right).
$$
It is not difficult to check that the supremum of $g(\eta)$ over $\eta \in [0,\infty)$ is given by 
$$
\alpha = \sup_{\eta\geq 0}\left( 1-\frac{2}{\sqrt{\pi}}\frac{\exp (-\eta^2)}{1+\erf (\eta)}\left[\frac{1}{\sqrt{\pi}}\frac{\exp (-\eta^2)}{1+\erf (\eta)}+\eta\right]\right)=1\,.
$$ 
Our assumptions on $\Phi$ and $W_0\leq 0$ imply that $\frac{\partial G}{\partial m}(m,\sigma)<0$, and then we obtain the desired unique zero of $G$ defining our stationary state $\fstat$ through \eqref{eq:stat}--\eqref{eq:spathomsol}.
\end{proof}

\begin{remark}\label{rem:regular}
Notice that the previous proposition can also be applied for functions $\Phi$ admitting negative values as $\frac{\exp(-\eta^2)}{1+\erf (\eta)}$ behaves like $-\sqrt{\pi}\eta$ in the limit $\eta \to - \infty$, such that one can show
$$
0\leq \left( 1-\frac{2}{\sqrt{\pi}}\frac{\exp (-\eta^2)}{1+\erf (\eta)}\left[\frac{1}{\sqrt{\pi}}\frac{\exp (-\eta^2)}{1+\erf (\eta)}+\eta\right]\right) \leq 1
$$
for any $\eta \in \R$.
For instance, the theorem is valid for the $\varepsilon$-approximation of $\Phi (x) = (x)^+$ defined by 
\begin{equation}\label{eq:modul}
\Phi_\varepsilon(x) = 0.5x\left(1 +\frac{x}{\sqrt{x^2+\varepsilon}}\right)    
\end{equation}
for $\varepsilon$ small enough such that $\Phi_\varepsilon'(x) \geq \frac{1}{W_0}$. The smooth approximation $\tilde \Phi_\varepsilon(x) = 0.5(x+\sqrt{x^2+\varepsilon})$ of $\Phi (x) = (x)^+$ can also be used as $\tilde \Phi_\varepsilon(x)$ trivially satisfies the hypotheses of Proposition \ref{prop:exisuniquemean} since $\tilde \Phi_\varepsilon(x)>0$, and it is strictly increasing. 
\end{remark}

\begin{figure}[ht]
\centering
\subfigure
{\includegraphics[trim={0cm 0.cm 0cm 0cm},clip, width=0.35\textwidth]{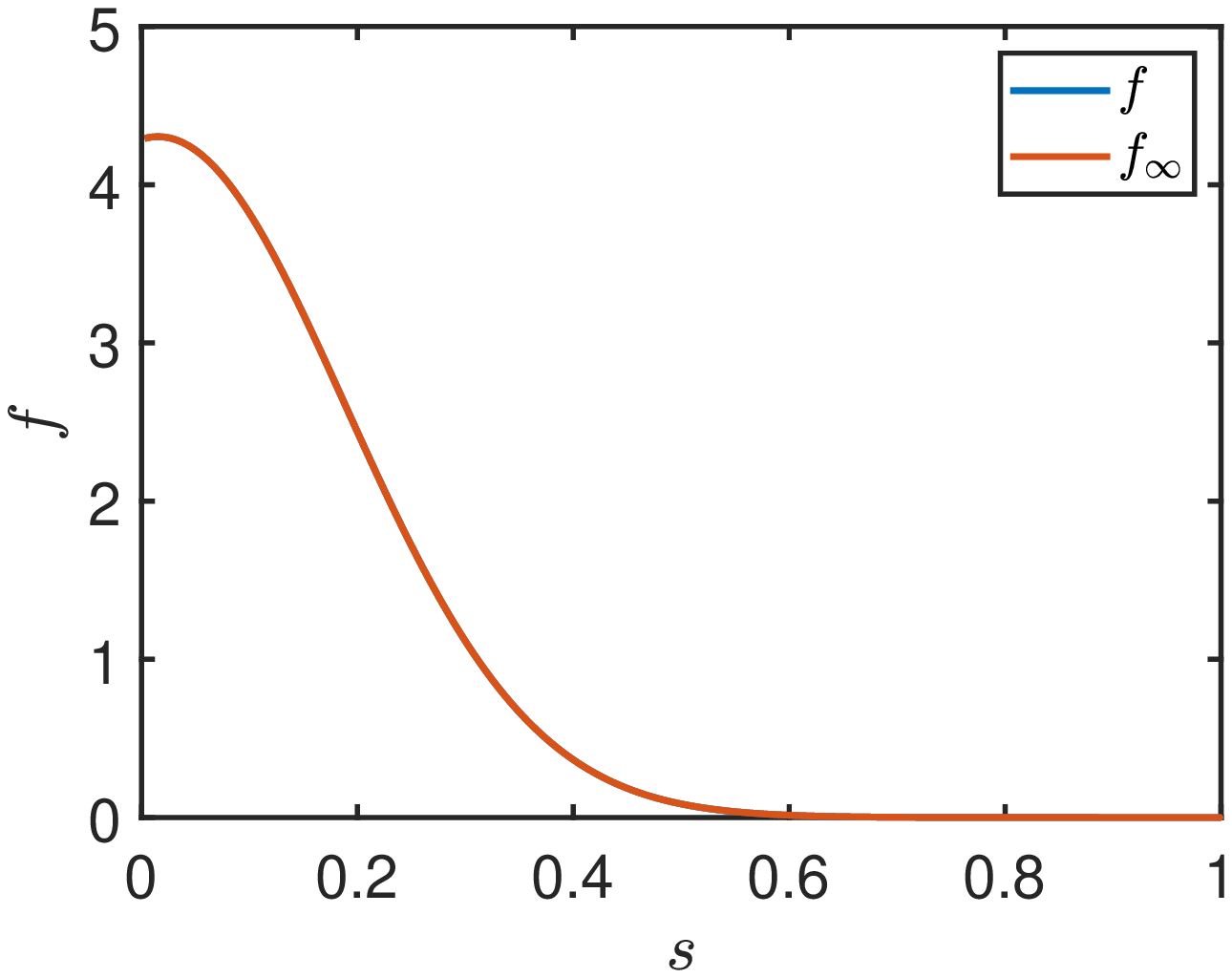}}
\subfigure
{\includegraphics[trim={0cm 0.cm 0cm 0cm},clip,width=0.35\textwidth]{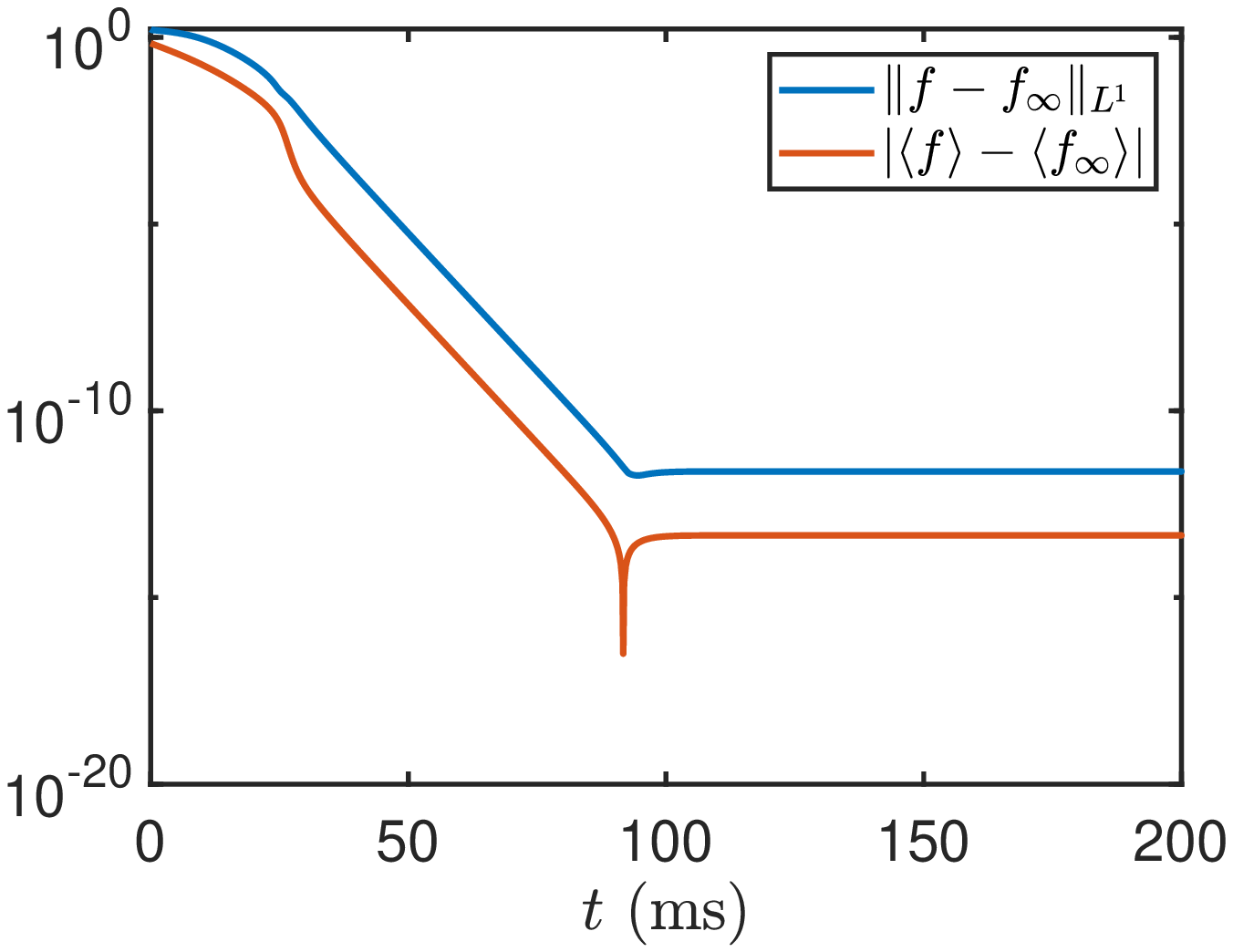}}
\caption{Stabilisation in time of the homogeneous problem. Left: numerical steady state $f$ (blue) versus the fixed point steady state $f_\infty$ (red) obtained from \eqref{eq:stat}--\eqref{eq:spathomsol} as functions of $s$. Right: the $L^1$-error and difference in mean between $f$ and $f_\infty$ plotted as functions of time $t$. Parameters: $\Phi_\eps$ with $\eps=0.01$. $W_0=-20.6711$. $\sigma=0.03$. The interval $[0,3]$ is split into $512$ grid points. Initial data: at random grid points $s_j, j=1,\dots,51$, $f_0(s_j)= 512/153$ and zero elsewhere. Average of the slopes over 100 runs: 0.78 (difference in mean) and 0.57 ($L^1$-difference).}
\label{fig:stability-hom}
\end{figure}

We have numerically analysed the stability of the stationary solutions obtained in the previous result among spatially homogeneous solutions of \eqref{eq:PDEhom}. In Fig.~\ref{fig:stability-hom}(a), we illustrate that the computed stationary state and the numerical solution to the evolution problem after time $t=150$ms are indistinguishable for the firing rate $\Phi_\eps$ with $\eps=0.01$. In Fig.~\ref{fig:stability-hom}(b), we observe the convergence in time towards the stationary state by computing the difference in $L^1$ and the difference in average between the stationary solution and the evolution problem. We conclude that the stationary state and the corresponding numerical solution to the evolution problem \eqref{eq:PDEhom} are identical to machine precision after $t=150$ms, and that the convergence in time is exponential with the rates computed by averaging over 100 runs.

Next, we focus on the linear stability of the spatially homogeneous solution as a solution of the nonlinear system \eqref{eq:PDE}--\eqref{eq:PDEbc}. For comparison, we first find the condition for linear spatial stability in the case of zero noise ($\sigma=0$) following the classical approach as in \cite{Murray03}. Let $\hat{W}(\mathbf{k})$ be the two-dimensional Fourier transform of $W$ restricted to $\Omega$, 
\begin{align*}
\hat{W}(\mathbf{k}) = \int_\Omega W(\bx)\exp(-i\mathbf{k}\cdot\bx) d\bx,  
\end{align*}
with $\mathbf{k} = \begin{pmatrix}
2\pi k_1 & 2\pi k_2
\end{pmatrix}^\top$, $k_1, k_2 \in \Z$.
As an example, if $W({\bx})$ is the characteristic of a ball with radius $r$ fully contained in $\Omega$, then
\begin{align*}
\hat{W}(\mathbf{k}) = 2 \pi \frac{J_1(r|\bf k|)}{|\bf k|},
\end{align*}
where $J_1(x)$ is a Bessel function.
Now let $\Phi_0' = \Phi'(W_0\langle \fstat \rangle+B)$ with $\fstat$ given by \eqref{eq:stat}, and define 
\begin{align}\label{eq:Fk}
F(\mathbf{k}) = \frac14\Phi_0'\hat{W}({\bf k})\sum_\beta \exp \big(-i \mathbf{k}\cdot \mathbf{r}^\beta\big). 
\end{align}
\begin{remark}\label{rem:shifts}
Given the assumptions on $W$ in \ref{assumptionW1} and \ref{assumptionW2}, the function $F(\bk )$ is real-valued when the shifts $\mathbf{r}^\beta$, $\beta=1,2,3,4$, satisfy the assumptions in \ref{assumptionrbeta}. With shifts of equal size $z$, one can check that
\begin{align*}
\sum_\beta \exp \big(-i \mathbf{k}\cdot \mathbf{r}^\beta \big) = 2 \cos (2 \pi k_1 z) + 2 \cos (2\pi k_2 z)
\end{align*}
with $\mathbf{k} = \begin{pmatrix}
2\pi k_1 & 2\pi k_2
\end{pmatrix}^\top$ using Euler's formula.
\end{remark}
The following lemma presents a linear stability condition for the system \eqref{eq:PDE} without noise. 

\begin{remark}
The proof of the mean field limit in \cite{CCS21} relies on $\sigma>0$. The mean field limit in the case $\sigma=0$ is easier to obtain and leads to a pure Vlasov equation. This is a very classical result in the smooth setting and is derived via estimates in transport distances, see \cite{HJ1, CCR11, J14, HJ2, Go16} and the references therein.
\end{remark}

\begin{lemma}\label{lem:nonoise}
Assume \ref{assumptionW1}--\ref{assumptionrbeta}. Let $F(\mathbf{k})$ be as in \eqref{eq:Fk}. Then the mean of the zero noise, i.e., $\sigma=0$, spatially homogeneous, stationary solution of \eqref{eq:PDE}--\eqref{eq:PDEbc} is linearly asymptotically stable if $F({\bf k}) < 1$.
\end{lemma}

\begin{proof}
By taking the mean of \eqref{eq:PDE}--\eqref{eq:PDEbc} with $\sigma = 0$, we find that the mean at position $\bx \in \Omega$, $s^\beta(t,\bx)=\langle f^\beta\rangle|_{\sigma=0}(t, \bx)$, evolves according to (dropping the $t$ dependence for ease of notation)
\begin{align} \label{eq:meanode}
\tau\frac{d}{dt} s^\beta(\bx)  = \Phi\left(\frac{1}{4}\sum_{\beta'}\int_\Omega W^{\beta'}(\bx-\by) s^{\beta'}(\by) d{\by} + B\right) -  s^\beta(\bx).
\end{align}
We linearise \eqref{eq:meanode} around the mean stationary, spatially homogeneous solution $s_\infty$, defined through $s_\infty = \Phi \big(W_0 s_\infty + B\big)$, and get 
\begin{align*}
\tau \frac{d}{dt} h^\beta (\bx) = \frac{\Phi_0'}{4}\sum_{\beta'}\int_\Omega W^{\beta'}(\bx-\by) h^{\beta'}(\by) d{\by} - h^\beta(\bx),
\end{align*}
where $h^\beta = s^\beta - s_\infty$ and $\Phi_0' = \Phi'(W_0s_\infty+B)$. Applying the ansatz $h^\beta(\bx,t) \propto \exp (\lambda t + i \bk \cdot \bx)$ to the equation, we find that each mode $\bk$ has the characteristic polynomial 
\begin{align*}
p_\bk(\lambda) = \left(\lambda + \frac1\tau \right)^3 \left(\frac{\Phi_0'\hat{W}(\mathbf{k})}{4\tau}\sum_\beta \exp \big(-i \mathbf{k}\cdot \mathbf{r}^\beta\big) - \frac1\tau -\lambda \right).
\end{align*}
We see that three of the eigenvalues are stable as long as $\tau >0$. The fourth eigenvalue is
\begin{align*}
\lambda = \frac{\Phi_0'\hat{W}(\mathbf{k})}{4\tau}\sum_\beta \exp \big(-i \mathbf{k}\cdot \mathbf{r}^\beta\big) - \frac1\tau = \frac1\tau \left( F(\mathbf{k})-1 \right),
\end{align*}
which determines whether the linear system is stable or not. The eigenvalue is negative if $F(\mathbf{k})<1$.
\end{proof}

We now turn to the case with noise, $\sigma>0$. Let $h^\beta = f^\beta - \fstat  $, where $\fstat$ is defined through \eqref{eq:stat}--\eqref{eq:spathomsol}. Notice that for the perturbations to be admissible, we need to ensure that $\int_0^\infty f^\beta ds = 1$, and consequently
$$
\int_0^\infty h^\beta ds = 0.
$$
In principle one needs $f^\beta \geq 0$ for it to be a probability density. However, we will prove below that linear stability holds without any assumption on the sign of $f^\beta$.

After linearising \eqref{eq:PDE} around the spatially homogeneous state $\fstat $, we get 
\begin{align}\label{eq:linearpde}
    \tau h_t^\beta =&\, -\partial_s \fstat  \frac{\Phi_0'}{4}\sum_{\beta'}\int_\Omega W^{\beta'}(\bx-\by)\langle h^{\beta'}\rangle d\by -\partial_s\big[(\Phi_0-s)h^\beta\big] + \sigma \partial_{ss}h^\beta, \\
    \Bigg(\fstat \frac{\Phi_0'}{4} &\sum_{\beta'}\int_\Omega W^{\beta'}(\bx-\by)\langle h^{\beta'}\rangle d\by +\Phi_0 h^\beta - \sigma \partial_{s}h^\beta\Bigg) \bigg|_{s=0} \!\!\!\!\!\!\!\!\!= 0,\nonumber
\end{align}
with $\beta =1, \dots, 4$.
By a change of variables $h^\beta = \fstat  v^\beta$, where $\fstat $ is the stationary state satisfying 
\begin{align*}
    -\partial_s\big(
[\Phi_0 -s] \fstat 
\big) + \sigma \partial_{ss}\fstat   = 0, \\
\big( \Phi_0 \fstat  - \sigma \partial_s \fstat  \big) \big|_{s=0} \!\! =0,
\end{align*}
we get 
\begin{align}\label{eq:pde-u1}
    \tau v_t^\beta = &\,  -\frac{(\Phi_0-s)}{\sigma}\frac{\Phi_0'}{4}\sum_{\beta'}\int_\Omega W^{\beta'}(\bx-\by)\langle\fstat  v^{\beta'}\rangle d\by +(\Phi_0-s)v_s^\beta + \sigma v_{ss}^\beta, \\
   \Bigg(\frac{\Phi_0'}{4}&\sum_{\beta'}\int_\Omega W^{\beta'}(\bx-\by)\langle\fstat  v^{\beta'}\rangle d\by -\sigma v_s^\beta \Bigg) \bigg|_{s=0}\!\!\!\!\!\!\!\!\! = 0,\nonumber
\end{align}
with $\beta =1, \dots, 4$. We now restrict the set of perturbations in $L^2(\Omega \times[0, \infty))$ to the ones of the form 
\begin{equation}\label{eq:fourierseries}
h^\beta(\bx,s,t) = \fstat(s)\sum_\bk  \exp (i \bk \cdot \bx) u^\beta_\bk(s,t),    
\end{equation}
where $u^\beta_\bk(s,t)$ is sufficiently smooth. We can then reduce \eqref{eq:pde-u1} to one Fourier mode. Dropping the subscript $\bk$ we set $v^\beta(\bx,s,t) = \exp (i \bk \cdot \bx)u^\beta(s,t)$, where $u^\beta$ may be complex-valued, and $U = \sum_\beta \exp(-i\bk \cdot {\bf r}^\beta ) u^\beta$, such that \eqref{eq:pde-u1} turns into
\begin{equation}\label{eq:pde-u}
\begin{aligned}
    \tau u_t^\beta = &\, -\frac{(\Phi_0-s)}{\sigma}\frac{\Phi_0'}{4}\hat{W}(\bk)\langle\fstat U\rangle +(\Phi_0-s)u_s^\beta + \sigma u_{ss}^\beta, \\
    &\quad \left(\frac{\Phi_0'}{4}\hat{W}(\bk)\langle\fstat U\rangle -\sigma u_s^\beta \right) \bigg|_{s=0} \!\!\!\!\!\!\!\!\!= 0,\quad 
\end{aligned}
\end{equation}
for $\beta =1, \dots, 4.$ An equation for the time evolution of $U$ can also derived, namely
\begin{equation}\label{eq:pde-U}
    \tau U_t= -\frac{(\Phi_0-s)}{\sigma}F(\bk)\langle\fstat U\rangle +(\Phi_0-s)U_s + \sigma U_{ss}, 
\end{equation}
where $F(\bk)$ is defined in \eqref{eq:Fk}.
Denote 
\begin{align}\label{eq:Minfty}
M_\infty= M_\infty(\sigma) = \int_0^\infty\big(s-\langle \fstat \rangle\big)^2\fstat ds,
\end{align}
where the dependence on $\sigma$ enters through $\fstat$ defined by \eqref{eq:stat}--\eqref{eq:spathomsol}.
We remark that we cannot obtain closed ODE equations for the moments of the distribution in the activity variable $s$, and thus a similar analysis as in \cite{TouboulPhysD,TouboulSIAM} is not possible here. Building on \eqref{eq:pde-u}, we can prove the following result. 

\begin{theorem}\label{thm:linearstability}
Assume \ref{assumptionW1}--\ref{assumptionrbeta}. Let $\Phi$ satisfy the assumptions in Proposition \ref{prop:exisuniquemean} and let $F(\mathbf{k})$ be as in \eqref{eq:Fk} and satisfy the condition in Remark \ref{rem:shifts}. Then the spatially homogeneous steady solution $f_\infty$ to \eqref{eq:PDE}--\eqref{eq:PDEbc} is linearly asymptotically stable in $L^2\big(\Omega \times [0,\infty)\big)$ for admissible perturbations of the form \eqref{eq:fourierseries} as long as
\begin{align}\label{eq:stabilitycondition} 
F(\mathbf{k}) < \frac{\sigma}{M_\infty} \qquad \mbox{for all \,} \bk .
\end{align} 
\end{theorem}
\begin{remark}
Using the relations \eqref{eq:stat}--\eqref{eq:spathomsol}, one can check after some tedious computations that $M_\infty(\sigma)$
satisfies $M_\infty(0)=0$ and $M_\infty'(0)=1$ assuming $\Phi_0>0$. Thus, $\frac{\sigma}{M_\infty} \to 1$ as $\sigma \to 0$, yielding the condition in Lemma \ref{lem:nonoise} in the zero noise limit.
\end{remark}

\begin{proof}[Proof of Theorem \ref{thm:linearstability}] The proof is split into three parts. First, we obtain an upper bound for the time derivative of $\int_0^\infty \fstat  |U|^2 ds$ (Part I). Recall that $U = \sum_\beta \exp(-i\bk \cdot {\bf r}^\beta ) u^\beta$. To obtain a time decaying estimate from the bound, we need to separate the linear part of $U$ from the nonlinear (Part II). Finally, we establish the stability of $U$, and consequently $u^\beta$, which then yields the asymptotic stability of $f_\infty$ (Part III).

{\bf Part I}: 
Note that 
\begin{align}\label{eq:zerointegral}
\int_0^\infty \fstat  U ds = \sum_\beta \exp(-i\bk \cdot {\bf r}^\beta ) \int_0^\infty h^\beta ds = 0,
\end{align}
such that 
\begin{align}\label{eq:dsfU}
 \int_0^\infty \partial_s \fstat  U ds =  \frac{1}{\sigma} \int_0^\infty (\Phi_0-s) \fstat U ds =  - \frac{1}{\sigma}\langle \fstat U \rangle.
\end{align} 
We multiply \eqref{eq:pde-U} with $\bar{U}$ (the complex conjugate of $U$), and integrate over $[0,\infty)$ with respect to $\fstat$:
\begin{align}
\frac{\tau}{2} \frac{d}{dt} \int_0^\infty \fstat  |U|^2 ds &= - F(\bk) \int_0^\infty\frac{\Phi_0-s}{\sigma}\fstat \bar{U} ds\,  \langle\fstat U \rangle\notag \\
&\quad + \int_0^\infty(\Phi_0-s)f_\infty U_s \bar{U}  ds + \sigma \int_0^\infty \fstat U_{ss} \bar{U} ds.\label{eq:intu2evolution}
\end{align}
After applying \eqref{eq:dsfU} to the first term on the right-hand side and integrating the last term by parts, we find that $U$ satisfies (where $\Re(z)$ denotes the real part of the complex number $z$)
\begin{align*}
\frac{\tau}{2} \frac{d}{dt} \int_0^\infty \fstat  |U|^2 ds & =  \sigma F(\bk) \left|\int_0^\infty \partial_s\fstat  U ds \right|^2 - \sigma \Re (\fstat \bar{U} \partial_s U) |_{s=0} - \sigma \int_0^\infty \fstat  |U_s|^2 ds \\
 &=  \sigma (F(\bk)-1) \left|\int_0^\infty \partial_s\fstat  U ds \right|^2 - \sigma \Re (\fstat \bar{U} \partial_s U) |_{s=0} \\
& \quad + \sigma \left|\int_0^\infty \partial_s\fstat  U ds \right|^2- \sigma \int_0^\infty \fstat  |U_s|^2 ds.
\end{align*}
Using the boundary condition on the second term above and by again utilizing the equivalence \eqref{eq:dsfU}, we get 
\begin{align*}
\frac{\tau}{2} \frac{d}{dt} \int_0^\infty \fstat  |U|^2 ds &= \sigma (F(\bk)-1) \left|\int_{R_+} \partial_s\fstat  U ds \right|^2 + \sigma  F(\bk)\Re \left(\big(\fstat \bar{U}\big)|_{s=0} \int_0^\infty \partial_s\fstat  U ds \right)  \\
& \quad+ \sigma \left|\int_0^\infty \partial_s\fstat  U ds \right|^2 - \sigma \int_0^\infty \fstat  |U_s|^2 ds \\     
&=  \frac{1}{\sigma}(F(\bk)-1) \left|\langle \fstat  U\rangle \right|^2 - F(\bk)\Re \left(\big(\fstat \bar{U}\big) |_{s=0} \langle \fstat  U\rangle \right) \\
&\quad +\sigma \left|\int_0^\infty \partial_s\fstat  U ds \right|^2 - \sigma \int_0^\infty \fstat  |U_s|^2 ds.
\end{align*}
Performing an integration by parts on the second to last integral and then yet again using \eqref{eq:dsfU},
\begin{align*}
 \left|\int_0^\infty \partial_s\fstat  U ds \right|^2 & =   \left|\big(\fstat U\big)|_{s=0}+ \int_0^\infty \fstat  U_s ds \right|^2 \\
 & = \big(\fstat ^2 |U|^2\big)|_{s=0} + 2 \Re \left(\big(\fstat \bar{U}\big)|_{s=0} \int_0^\infty \fstat  U_s ds\right) + \left| \int_0^\infty \fstat  U_s ds \right|^2 \\
 & = - \big(\fstat ^2 |U|^2\big)|_{s=0} +  \frac{2}{\sigma} \Re \left(\big(\fstat \bar{U}\big)|_{s=0} \langle \fstat  U \rangle \right) + \left| \int_0^\infty \fstat  U_s ds \right|^2,
\end{align*}
we arrive at
\begin{align*}
\frac{\tau}{2} \frac{d}{dt} \int_0^\infty \fstat  |U|^2 ds &=  \frac{1}{\sigma}(F(\bk)-1) \left|\langle \fstat  U\rangle \right|^2 +(2- F(\bk)) \Re \left(\big(\fstat \bar{U}\big)|_{s=0} \langle \fstat  U\rangle \right) 
- \sigma \big(\fstat^2 |U|^2\big)|_{s=0} \\
&\quad +\sigma \left|\int_0^\infty \fstat  U_s ds \right|^2 - \sigma \int_0^\infty \fstat  |U_s|^2 ds.
\end{align*}
From \eqref{eq:pde-u} and the definition of $U$, it can be shown that $\langle \fstat U \rangle$ follows
\begin{align*}
\frac{\tau}{2} \frac{d}{dt} |\langle \fstat  U \rangle|^2 = \big(F(\bk)-1 \big) |\langle \fstat  U \rangle|^2 + \sigma \Re \left( \big(\fstat \bar{U}\big)|_{s=0}\langle \fstat  U \rangle \right),
\end{align*}
with $F(\bk)$ as in \eqref{eq:Fk}. We add the two equalities,
\begin{align*}
\frac{\alpha \tau}{2}& \frac{d}{dt} \int_0^\infty  \fstat  |U|^2 ds + \frac{\beta \tau}{2\sigma} \frac{d}{dt} |\langle \fstat  U \rangle|^2 \\
&= \frac{\alpha + \beta}{\sigma} \big(F(\bk)-1 \big) |\langle \fstat  U \rangle|^2 + \left( \alpha (2-F(\bk)) + \beta  \right)\Re \left( \big(\fstat \bar{U}\big)|_{s=0}\langle \fstat  U \rangle \right) 
- \alpha \sigma \big(\fstat^2 |U|^2\big)|_{s=0} \\
&\quad + \alpha  \sigma \left|\int_0^\infty \fstat  U_s ds \right|^2 - \alpha \sigma \int_0^\infty \fstat  |U_s|^2 ds.
\end{align*}
By setting $\alpha = 1$ and $\beta = - F(\bk)$, we get
\begin{align}\label{eq:equality}
\frac{\tau}{2} \frac{d}{dt}\left( \int_0^\infty \fstat  |U|^2 ds - \frac{F(\bk)}{\sigma} |\langle \fstat  U \rangle|^2 \right) = 
& - \left|\left(F(\bk)-1 \right) \frac{1}{\sqrt{\sigma}} \langle \fstat  U \rangle 
+\sqrt{\sigma}\big(\fstat U\big)\big|_{s=0} \right|^2 \\
&  +\sigma \left|\int_0^\infty \fstat  U_s ds \right|^2 - \sigma \int_0^\infty \fstat  |U_s|^2 ds. \nonumber
\end{align}
In the above derivation we have used that
\begin{align*}
2(1-F(\bk))\Re \left( \big(\fstat \bar{U}\big)|_{s=0}\langle \fstat  U \rangle \right) = (1-F(\bk)) \left(\big(\fstat \bar{U}\big)|_{s=0}\langle \fstat  U \rangle + \big(\fstat U\big)|_{s=0}\langle \fstat  \bar{U} \rangle\right).
\end{align*}
By applying the Cauchy--Schwarz inequality to $ \left|\int_0^\infty \fstat  U_s ds \right|^2$ with $u_1 = \fstat^\hf, u_2 = \fstat^\hf U_s$, we see that the right-hand side of \eqref{eq:equality} is non-positive. However, it is not straightforward to determine whether the right-hand side is strictly negative or if there is a Gr\"onwall type decay estimate to be obtained from this expression. In particular, the Cauchy--Schwarz inequality with the chosen functions $u_1 = \fstat^\hf, u_2 = \fstat^\hf U_s$ is an equality, 
$$
\left|\int_0^\infty \fstat  U_s ds \right|^2 = \Big(\int_0^\infty \fstat ds \Big) \Big( \int_0^\infty \fstat  |U_s|^2 ds\Big) = \int_0^\infty \fstat  |U_s|^2 ds,
$$ 
whenever $u_2 = c(t) u_1$ for any function $c(t)$, which here means that $U = c(t) (s- \langle \fstat \rangle)$ when we make sure that the requirement $\int_0^\infty U \fstat ds = 0$ holds.

{\bf Part II}: To separate the linear part from the nonlinear part, we split $U$ into $U = V + c(t)(s-\langle \fstat \rangle)$.
We have that $\int_0^\infty V \fstat ds = 0$. Next, we choose $c(t)$ such that $V$ and $s-\langle \fstat \rangle$ are orthogonal with respect to the measure $\fstat ds$, i.e.,  $\int_0^\infty V (s-\langle \fstat \rangle) \fstat ds = 0$.  For this we choose
\begin{equation}
    c(t)=\langle\fstat U\rangle \left(\int_0^\infty s^2\fstat(s) ds-\langle\fstat\rangle^2  \right)^{-1}=\frac{\langle\fstat U\rangle}{M_\infty},
\end{equation}
where $M_\infty$ is defined in \eqref{eq:Minfty}. We can summarise this as
\begin{align*}
\int_0^\infty \fstat  |U|^2 ds & = \int_0^\infty \fstat  |V|^2 ds + M_\infty |c(t)|^2, \\
\langle \fstat U \rangle & = M_\infty c(t), \\
(\fstat U) |_{s=0} & = \fstat(0)(V(0) -c(t)\langle \fstat \rangle),
\end{align*}
such that \eqref{eq:equality} becomes
\begin{gather} \label{eq:ortheq}
\begin{aligned}
\frac{\tau}{2} \frac{d}{dt} & \left( \int_0^\infty \fstat  |V|^2 ds + \left( 1 - M_\infty \frac{F(\bk)}{\sigma} \right) M_\infty |c(t)|^2 \right) \\
&=  - \left|\left(F(\bk)-1 \right) \frac{1}{\sqrt{\sigma}} M_\infty c(t) +\sqrt{\sigma}\fstat (V - c(t)\langle \fstat \rangle )\big|_{s=0} \right|^2  \\
&\quad+ \sigma \left|\int_0^\infty \fstat  V_s ds \right|^2 -  \sigma \int_0^\infty \fstat  |V_s|^2 ds. 
\end{aligned}
\end{gather}

From the orthogonality $\int_0^\infty V (s-\langle \fstat \rangle) \fstat ds = 0$, we get
\begin{align*}
(\fstat V) \big|_{s=0} & = -\int_0^\infty \partial_s (\fstat V) ds = -\int_0^\infty \partial_s \fstat V ds -\int_0^\infty \fstat V_s ds \\
& = \int_0^\infty \frac{s-\Phi_0}{\sigma} \fstat V ds -\int_0^\infty \fstat V_s ds = -\int_0^\infty \fstat V_s ds.
\end{align*}

Inserting this into the first square in \eqref{eq:ortheq} and then expanding the square, \eqref{eq:ortheq} turns into
\begin{align*}
\frac{\tau}{2} \frac{d}{dt} & \left( \int_0^\infty \fstat  |V|^2 ds + \left( 1 - M_\infty \frac{F(\bk)}{\sigma} \right) M_\infty |c(t)|^2 \right)\\
&\qquad\qquad=  - R^2|c(t)|^2 + 2 \sqrt{\sigma} R\, \Re \left( c(t) \int_0^\infty V_s \fstat ds \right) 
- \sigma \int_0^\infty \fstat  |V_s|^2 ds \\
&\qquad\qquad=  \int_0^\infty \left(- R^2|c(t)|^2 +2 \sqrt{\sigma} R \,\Re \left( c(t) V_s \right) -\sigma |V_s|^2 \right)
\fstat ds \\
&\qquad\qquad=  - \int_0^\infty\left| Rc(t) - \sqrt{\sigma} V_s \right|^2 \fstat ds,
\end{align*}
with $R = (F(\bk)-1)\frac{M_\infty}{\sqrt{\sigma}}-\sqrt{\sigma}\fstat(0)\langle \fstat \rangle$. 
We now apply the Poincar\'e inequality with respect to the measure $\fstat (s) ds$ \cite{Mu72,RBI17}, given by 
\begin{align*}
C \int_0^\infty \left|G- \int_0^\infty G \fstat ds \right|^2 \fstat  ds \leq \int_0^\infty |G_s |^2 \fstat ds.
\end{align*}
Let $G_s = Rc(t)-\sqrt{\sigma}V_s$. Then $G = Rc(t)(s-\langle \fstat\rangle)-\sqrt{\sigma}V$ and $\int_0^\infty G \fstat ds = 0$, such that 
\begin{gather}\label{eq:finalineq}
\begin{aligned}
\frac{\tau}{2} \frac{d}{dt} \left( \int_0^\infty \fstat  |V|^2 ds + \left( 1 - M_\infty \frac{F(\bk)}{\sigma} \right) M_\infty |c(t)|^2 \right) & \leq  - C \int_0^\infty\left| Rc(t)(s-\langle \fstat\rangle)-\sqrt{\sigma}V \right|^2 \fstat ds \\
&=  - C R^2 M_\infty |c(t)|^2- C\sigma \int_0^\infty|V |^2 \fstat ds.
\end{aligned}
\end{gather}
In the above calculation, we have used the orthogonality of $V$ and $c(t)(s-\langle \fstat \rangle)$ with respect to $\fstat ds$.

{\bf Part III}: Defining 
\begin{align*}
Q(t) := \int_0^\infty \fstat  |V|^2 ds, \quad
D(t) := M_\infty H |c(t)|^2, \quad H:=1 - M_\infty\frac{F(\bk)}{\sigma},
\end{align*}
the inequality \eqref{eq:finalineq} reads
\begin{align*}
\frac{d}{dt} (Q(t) + D(t)) \leq - \frac{2}{\tau} C \frac{R^2}{H} D(t) - \frac{2}{\tau} C \sigma Q(t) \leq
- \frac{2}{\tau} C \min \left\{\frac{R^2}{H}, \sigma \right\} \left( Q(t) + D(t)\right),
\end{align*}
due to \eqref{eq:stabilitycondition}. Note that also due to \eqref{eq:stabilitycondition}, $D(t) \geq 0$. It can be checked after some tedious computations using the explicit expression of $\fstat$ in \eqref{eq:stat}--\eqref{eq:spathomsol} that $R<0$ when $\Phi\geq 0$. This leads to the exponential decay by an application of Gr\"onwall's inequality 
\begin{align}\label{eq:expdecay}
Q(t) + D(t) \leq \left(Q(0) + D(0)\right) \exp \left(- \frac{2}{\tau} C \min \left\{\frac{R^2}{H}, \sigma \right\} t\right).
\end{align}
Thus, we can conclude that $U$ is asymptotically stable. What remains to show is that the same holds for $u^\beta$. We multiply \eqref{eq:pde-u} with $\bar{u}^\beta$ and integrate over $[0,\infty)$,
\begin{align*}
\frac{\tau}{2} \frac{d}{dt} \int_0^\infty \fstat  |u^\beta|^2 ds &= - \frac{\Phi_0'}{4}\hat{W}(\bk) \int_0^\infty\frac{\Phi_0-s}{\sigma}\fstat \bar{u}^\beta ds\,  \langle\fstat U \rangle\notag \\
&\quad + \int_0^\infty(\Phi_0-s)f_\infty u^\beta_s \bar{u}^\beta  ds + \sigma \int_0^\infty \fstat u^\beta_{ss} \bar{u}^\beta ds.
\end{align*}
As done for $U$, we integrate the last term by parts and use the boundary condition such that
\begin{align*}
\frac{\tau}{2} \frac{d}{dt} \int_0^\infty \fstat  |u^\beta|^2 ds &= - \frac{\Phi_0'}{4}\hat{W}(\bk)\left( \int_0^\infty\partial_s\fstat \bar{u}^\beta ds + \bar{u}^\beta\fstat|_{s=0}\right)\,  \langle\fstat U \rangle
  - \sigma \int_0^\infty \fstat |u^\beta_{s}|^2 ds \\
  &= \frac{\Phi_0'}{4}\hat{W}(\bk)\left( \int_0^\infty\fstat \bar{u}_s^\beta ds\right)\,  \langle\fstat U \rangle
  - \sigma \int_0^\infty \fstat |u^\beta_{s}|^2 ds \\
  & \leq \frac{|\Phi_0'|}{4}|\hat{W}(\bk)|\left(\frac{1}{2\alpha}|\langle\fstat U \rangle|^2+\frac{\alpha}{2} \left|\int_0^\infty\fstat \bar{u}_s^\beta ds\right|^2\right)
  - \sigma \int_0^\infty \fstat |u^\beta_{s}|^2 ds,
\end{align*}
where $\alpha$ is to be determined. We apply the Cauchy--Schwarz inequality to the middle term and rearrange,
\begin{align*}
\frac{d}{dt} \int_0^\infty \fstat  |u^\beta|^2 ds & \leq  \underbrace{\frac{|\Phi_0'|}{4\tau \alpha}|\hat{W}(\bk)|}_{=:\,C_1 }|\langle\fstat U \rangle|^2
-\underbrace{\left(-\alpha \frac{|\Phi_0'|}{4\tau}|\hat{W}(\bk)|
  + \frac{2}{\tau} \sigma \right)}_{=:C_2} \int_0^\infty \fstat |u^\beta_{s}|^2 ds.
\end{align*}
We now choose $\alpha > 0$ such that $C_2 > 0$, and then apply Poincar\'e's inequality to the integral, 
\begin{align*}
\frac{d}{dt} \int_0^\infty \fstat  |u^\beta|^2 ds & \leq  C_1 |\langle\fstat U \rangle|^2
-C_2 \int_0^\infty \fstat |u^\beta|^2 ds \\
& \leq \tilde{C}_1 \exp{(-C_3 t)} -\tilde{C_2} \int_0^\infty \fstat |u^\beta|^2 ds,
\end{align*}
where the exponential decay \eqref{eq:expdecay} is applied in the last step. This leads to 
\begin{align*}
\int_0^\infty \fstat  |u^\beta|^2 (t) ds \leq \frac{\tilde{C}_1}{\tilde{C_2}-C_3} \left( \exp (-C_3 t) - \exp (-\tilde{C_2} t) \right) +  \int_0^\infty \fstat  |u^\beta|^2 (0) ds \exp (-\tilde{C_2} t).  
\end{align*}
One can avoid $\tilde{C_2} = C_3$ by choosing $\alpha$ appropriately. 

The asymptotic stability of $f_\infty$ in $L^2\big(\Omega \times [0,\infty)\big)$ for the set of perturbations given by \eqref{eq:fourierseries} now follows by an application of Parseval's identity with respect to $\bx$ to 
$$
\int_0^\infty \int_\Omega|h^\beta|^2 d\bx ds\leq \tfrac1Z \int_0^\infty \int_\Omega|h^\beta|^2 d\bx \frac1{f_\infty(s)} ds,
$$ 
the identity \eqref{eq:fourierseries}, and the estimate above. Notice that $f_\infty(s)\leq \tfrac1Z$ from \eqref{eq:stat2}.
\end{proof} 

\begin{remark}
In principle, the linear stability analysis is valid only for smooth $\Phi$. However, the stability condition \eqref{eq:stabilitycondition} of the linearised problem does only depend on $\Phi'$ such that the result holds for the linearised system with $\Phi (x) = (x)^+$. Notice that condition \eqref{eq:stabilitycondition} is continuous with respect to the regularisation parameter $\varepsilon$ in Remark \ref{rem:regular} for which the linearisation is valid. 

We also remark that the value of $R$ in the proof above remains strictly negative as long as $\varepsilon$ is small enough despite the fact that $\Phi_\varepsilon$ may give negative values. 
\end{remark}

\begin{figure}[ht]
\centering
\subfigure
{
\includegraphics[trim={0cm 0.cm 0cm 0cm},clip, width=0.35\textwidth]{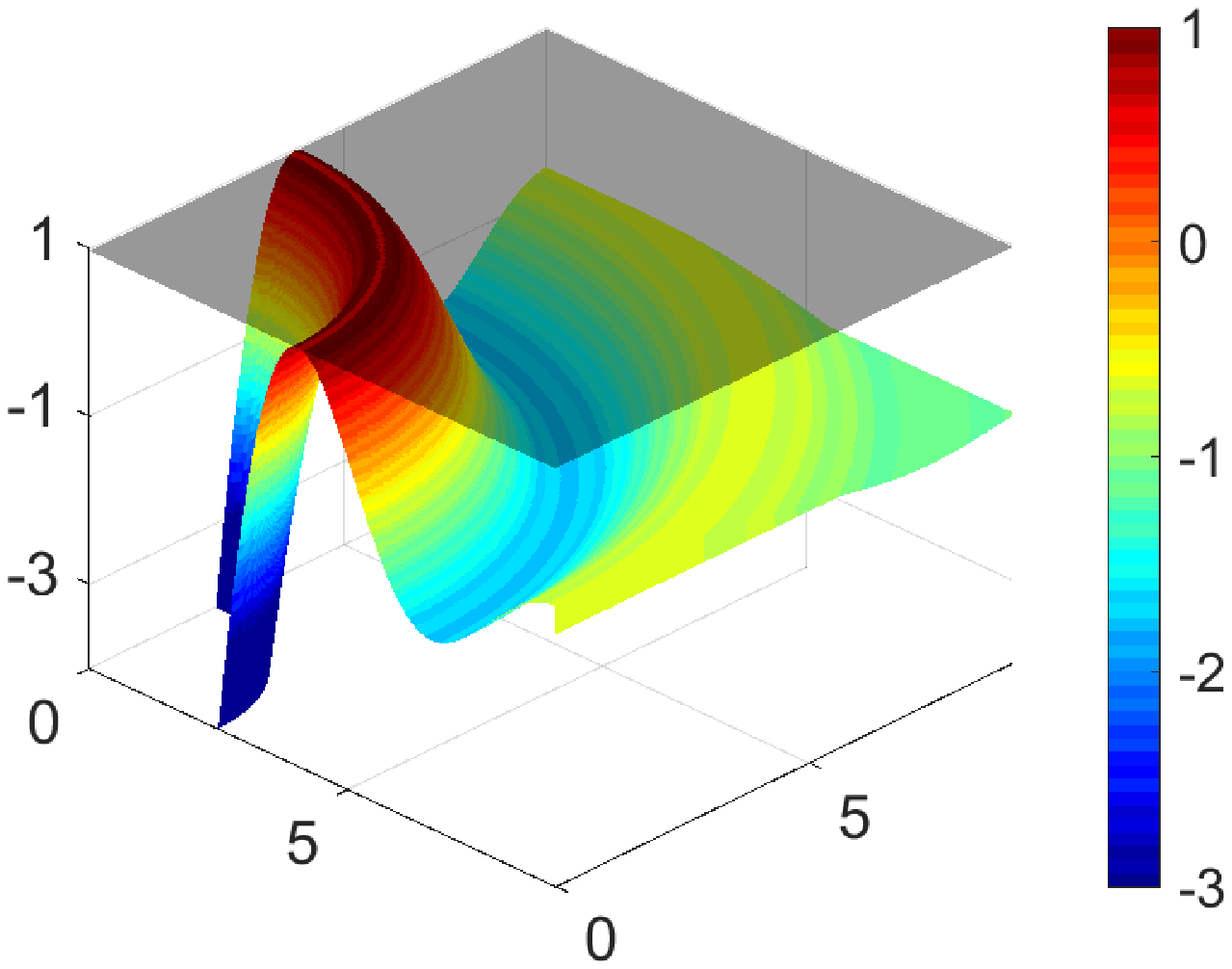}}
\subfigure
{
\includegraphics[trim={0cm 0.cm 0cm 0cm},clip,width=0.35\textwidth]{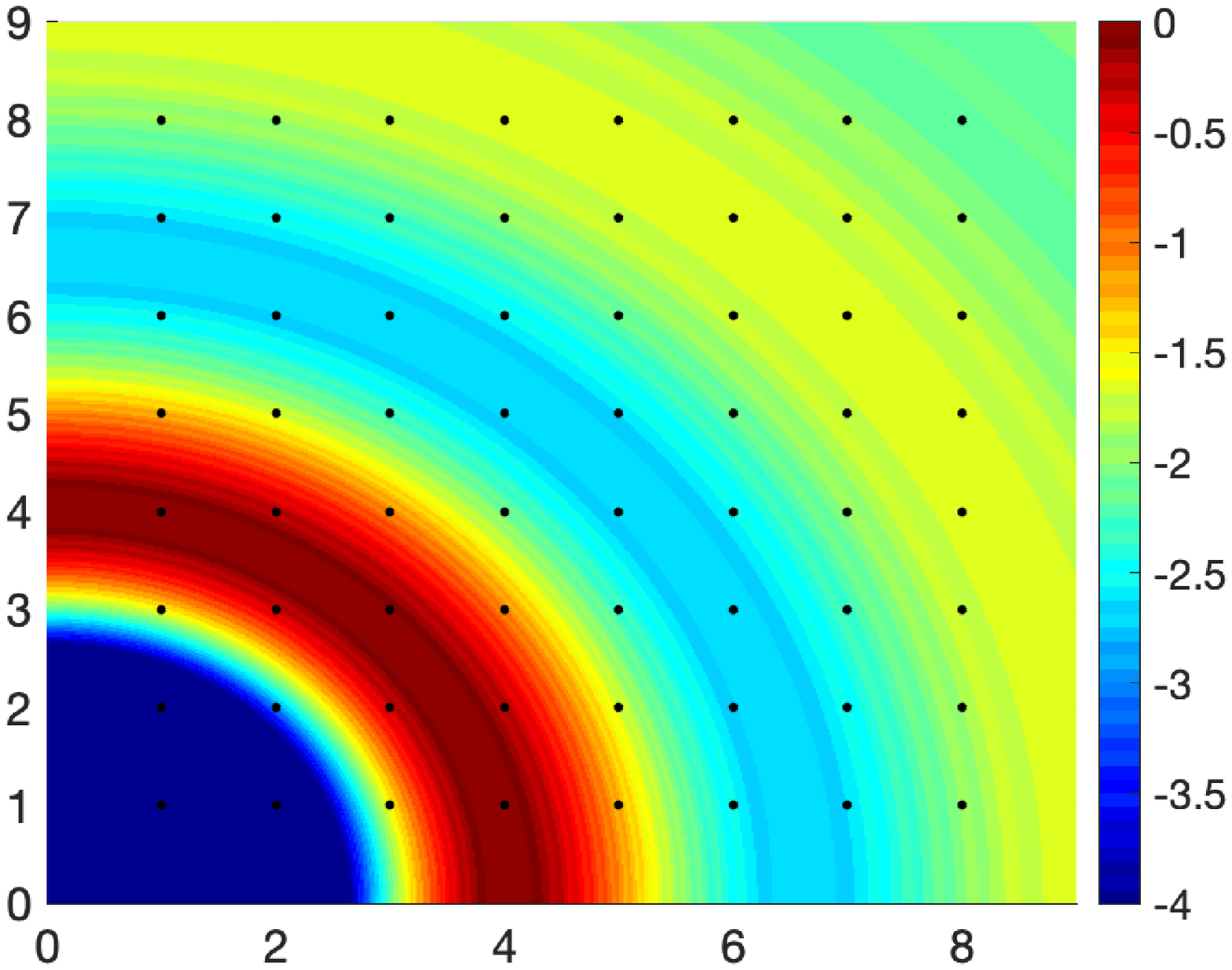}}
\caption{Plots of the linear stability condition on $\frac{M_\infty}{\sigma}F(\bk)$ for $\Phi(x) = \Phi_\eps(x), \eps =0.01$, with $W(|\bx|)=-0.005\cdot 128^2\left(1+\tanh(10-50|\bx|)\right)$. Left: $\frac{M_\infty}{\sigma}F(\bk)$ at the minimal value of the noise $\sigma$ for linear stability \eqref{eq:stabilitycondition}. Right: Associated contour plot highlighting the Fourier modes $\bk=2\pi\begin{pmatrix}
 k_1 & k_2
\end{pmatrix}^\top$ with black dots at $(k_1,k_2)$.}
\label{fig:Fk}
\end{figure}

\begin{figure*}
\centering
\subfigure
{
\includegraphics[trim={0cm 0.cm 0cm 0cm},clip, width=0.25\textwidth]{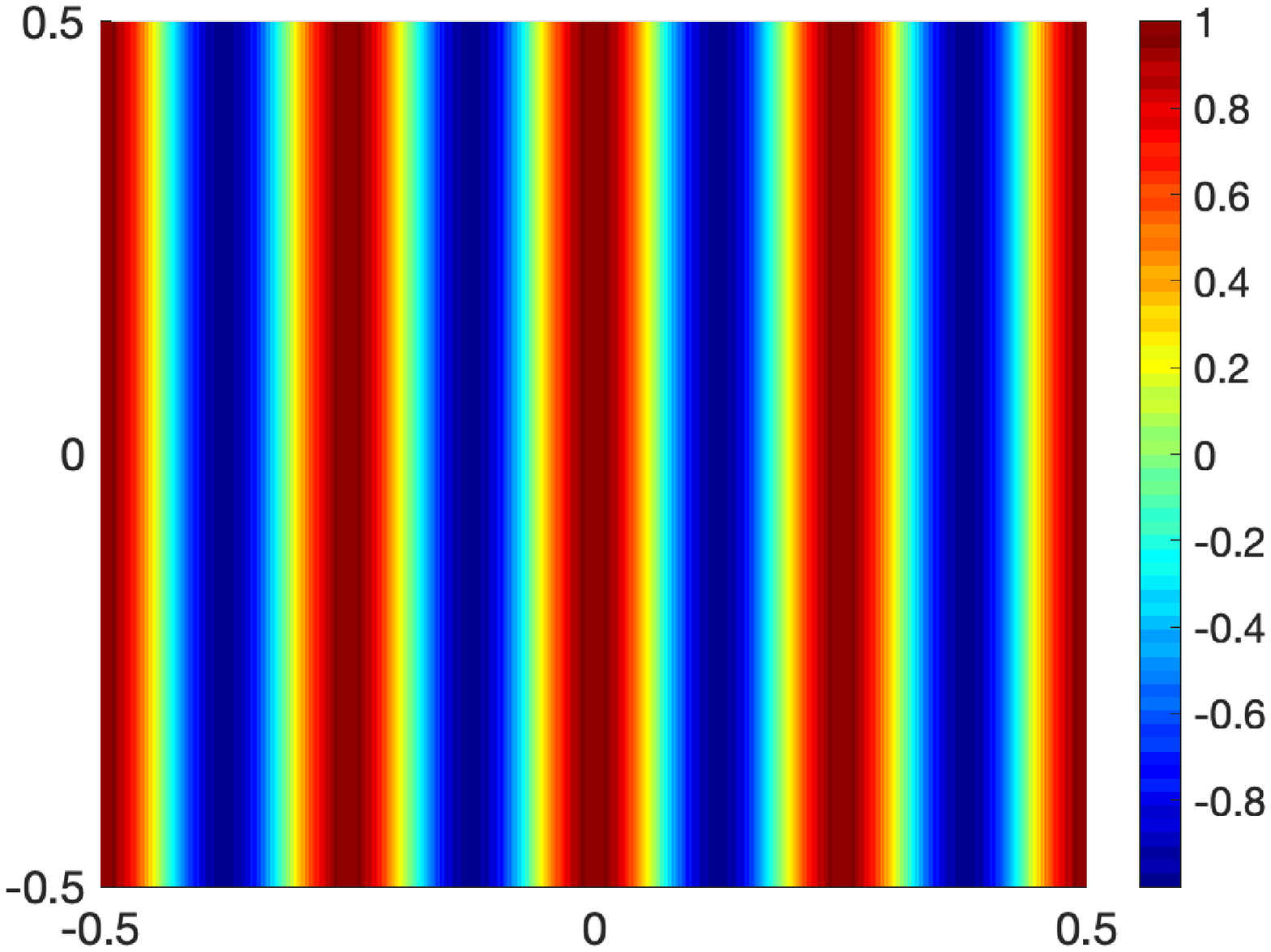}}
\subfigure
{
\includegraphics[trim={0cm 0.cm 0cm 0cm},clip,width=0.25\textwidth]{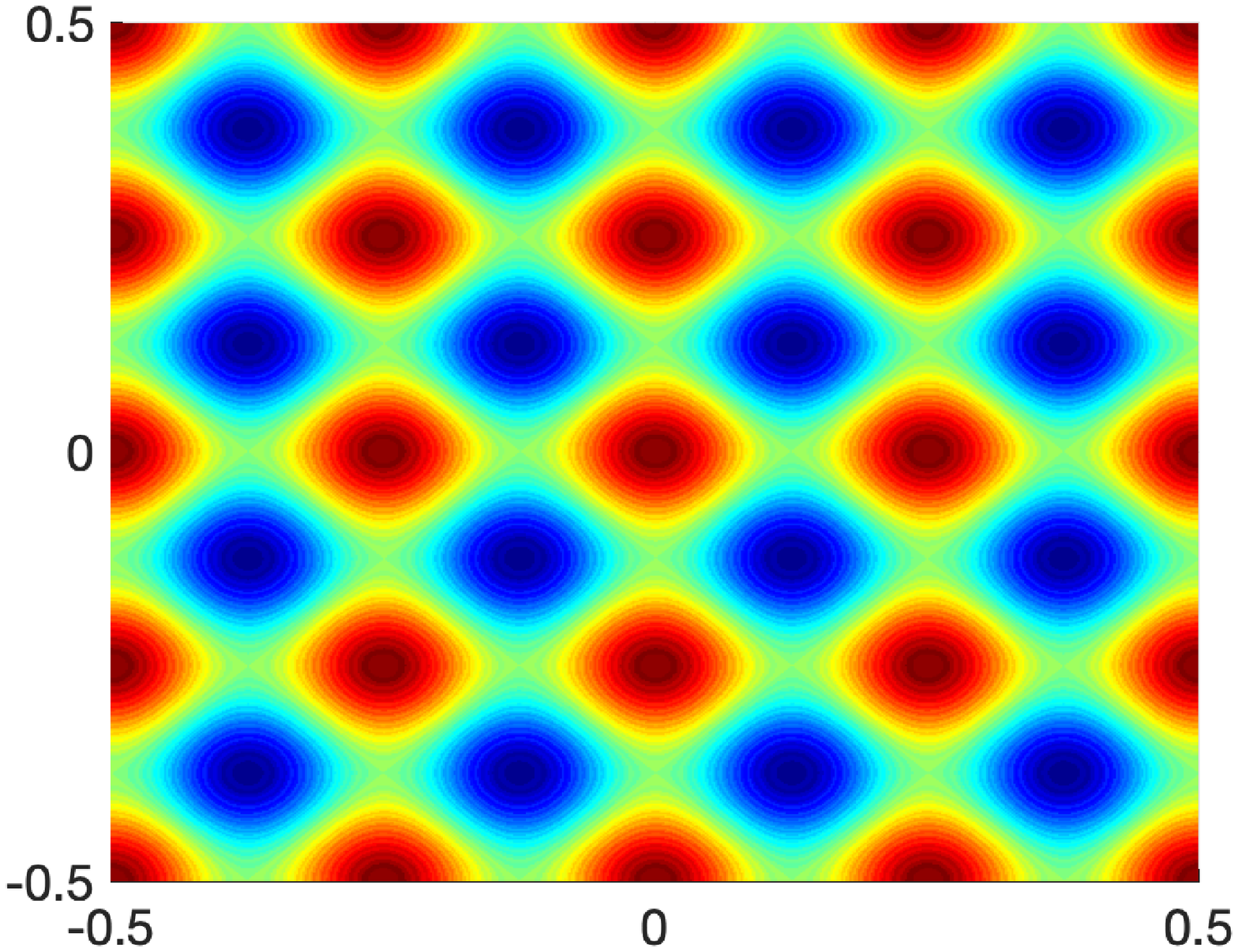}}
\subfigure
{
\includegraphics[trim={0cm 0.cm 0cm 0cm},clip,width=0.25\textwidth]{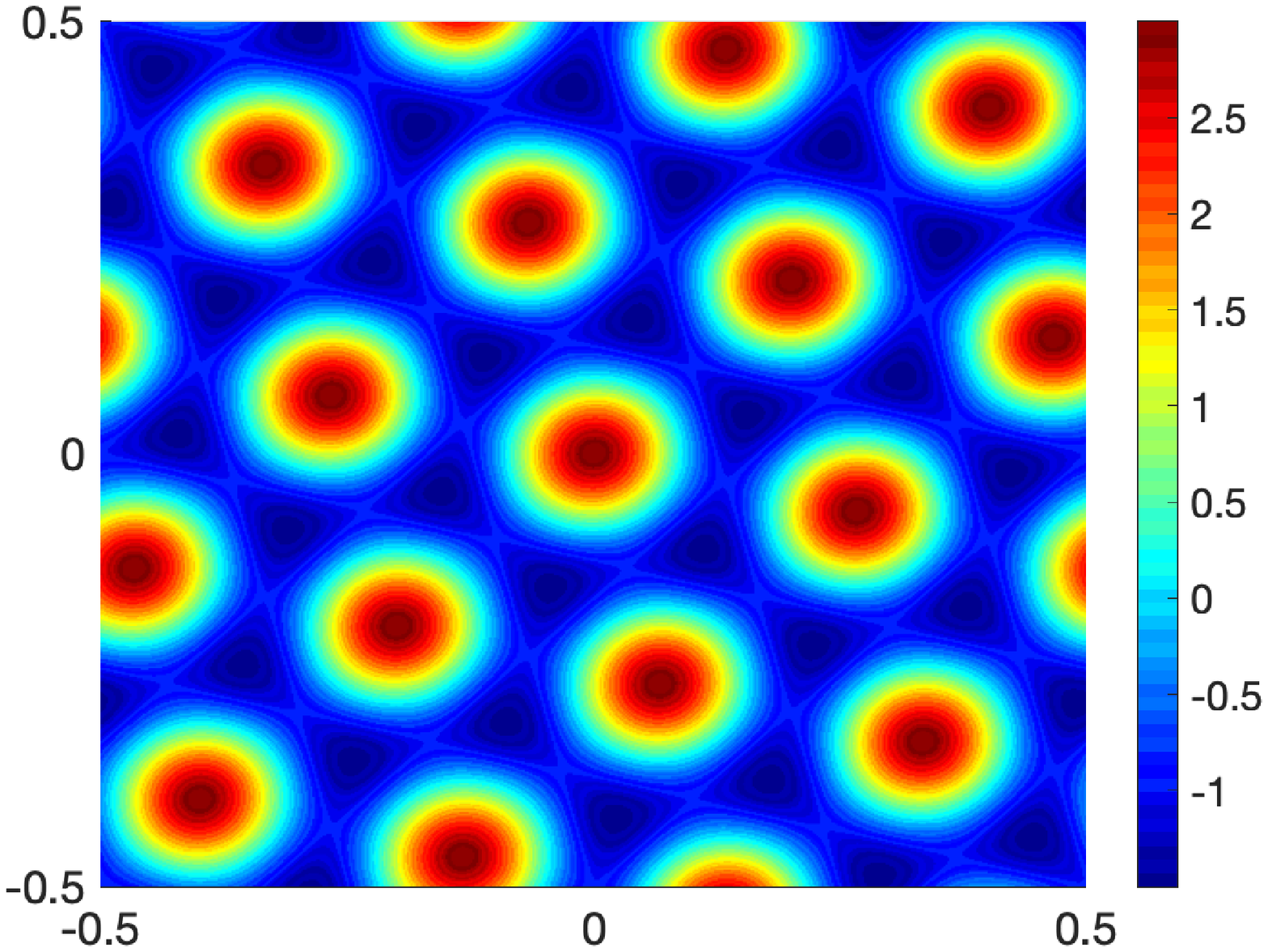}}
\caption{Linear combinations of $\cos (\bk\cdot\bx )$ are plotted against $x$ horizontally and $y$ vertically. The Fourier modes $\bk$ are chosen to be the points with the largest values of $F(\bk )$ in Fig.~\ref{fig:Fk}(b) with $W(|\bx|)=-0.005\cdot 128^2\left(1+\tanh(10-50|\bx|)\right)$. From left to right: ${\bf k} =2 \pi (4,0)$, $ {\bf k} = 2\pi(4,0), 2\pi(0,4)$ and $ {\bf k} =2\pi(4,1),2\pi(1,4),2\pi(3,-3)$.}
\label{fig:k-patterns}
\end{figure*}

\section{Bifurcation diagrams and phase transitions}
With our linear stability analysis at hand, we will now investigate how the stationary patterns of the nonlinear PDE system \eqref{eq:PDE}--\eqref{eq:PDEbc} change as we vary the noise parameter $\sigma$. We do this by numerically computing bifurcation branches from the spatially homogeneous solution $f_\infty$ for various choices of the modulation function $\Phi$. The numerical procedure is described in more detail in Section \ref{sec:bif}. 

\subsection{Instability of the linearised system}
First, to connect the linear stability analysis in Section \ref{sec:stability} with the patterns we observe for the full system \eqref{eq:PDE}--\eqref{eq:PDEbc}, we start by investigating the dominant Fourier modes $\bk$ of the perturbations $h^\beta(\bx,s,t) = \fstat(s)\sum_\bk  \exp (i \bk \cdot \bx) u^\beta_\bk(s,t)$, where $u^\beta_\bk$ satisfies \eqref{eq:pde-u}. The stability condition  \eqref{eq:stabilitycondition} of Theorem \ref{thm:linearstability} is visualised by plotting the function $\frac{M_\infty}{\sigma}F(\bk )$ for the case $\Phi(x) = \Phi_\eps(x), \eps =0.01$ (cf.~\eqref{eq:modul}), with $W(|\bx|)=-0.005\cdot 128^2\left(1+\tanh(10-50|\bx|)\right)$, against the modes $\mathbf{k}$ in Fig.~\ref{fig:Fk}. The figure illustrates that the maximum points over the lattice $\mathbf{k} = \begin{pmatrix} 2\pi k_1 & 2\pi k_2 \end{pmatrix}^\top$, $k_1,k_2\in \Z$ for this particular $W$ are among $\{(4,0),(4,1),(3,3),(1,4),(0,4)\}$ and their reflections by symmetries with respect to the origin, $k_1=0$, and $k_2=0$. Note that the modulation function $\Phi$ of the firing rate has no effect on the maximum points of $F(\bk)$ since it enters through as an amplification factor in \eqref{eq:Fk}. As a consequence, we may expect that the patterns leading the instability of the homogeneous in space stationary state $\fstat$ are driven by a  combination of these Fourier modes. Examples of possible patterns generated as a sum of cosines depending on the dominant modes, i.e., the maximum points of $F(\bk)$ over the lattice $\mathbf{k} = \begin{pmatrix} 2\pi k_1 & 2\pi k_2 \end{pmatrix}^\top$, $k_1,k_2\in \Z$, are depicted in Fig.~\ref{fig:k-patterns}. Notice that the rightmost plot displays a hexagonal pattern similar to the ones generated by the nonlinear PDE system \eqref{eq:PDE}--\eqref{eq:PDEbc} in the top and middle row of Fig.~\ref{fig:single}. See a similar strategy to this for a related problem in \cite[Ch.~12]{Murray03}.

\subsection{Bifurcations and phase transitions of the nonlinear PDE system}\label{sec:bif}
We now continue with our examination of the stability of the spatial patterns, generated by the full system \eqref{eq:PDE}--\eqref{eq:PDEbc}, with respect to $\sigma$. In Fig.~\ref{fig:bif-zoom}, we numerically compute bifurcation branches from the spatially homogeneous solution $\fstat$ for different modulation functions $\Phi$ with $W(|\bx|)=-0.005\cdot 128^2\left(1+\tanh(10-50|\bx|)\right)$ for the nonlinear problem \eqref{eq:PDE}--\eqref{eq:PDEbc}. This is done by using a continuation based method on $\sigma$ over an accurate numerical solver for the evolution in time of Fokker--Planck like equations developed in \cite{CCH}; further details are given in Appendix \ref{app:numerics}. The continuation method starts either at the largest or the smallest noise value $\sigma$ of the interval under consideration and it solves for the evolution in time of \eqref{eq:PDE}--\eqref{eq:PDEbc} up to stabilisation to a steady value. This allows for recursive computation of the stationary states for smaller or larger values of the noise by taking as initial data the already computed steady state. With this procedure we ensure, up to numerical accuracy, that we compute the stable stationary states, either by sweeping the noise values from left-to-right (l2r) or from right-to-left (r2l).

Each subplot in Fig.~\ref{fig:bif-zoom} shows the maximum and minimum over space $\bx$ of the average activity rate $\langle f\rangle(\bx)=\sum_\beta \langle f^\beta \rangle (\bx) $ of the computed steady states for each noise value $\sigma$. We show both the spatial maximum and minimum of $\langle f\rangle(\bx)$ to illustrate the fact that the computed stationary states are not spatially homogeneous, in other words, that they lead to spatial patterns. We also plot the spatially homogeneous branch numerically solving the implicit expression \eqref{eq:spathomsol} as reference. The red dots indicate the stability threshold in $\sigma$ for the condition $F(\bk)< 1$, as in Lemma \ref{lem:nonoise}, to hold.

In Fig.~\ref{fig:bif-zoom}(a), we observe the bifurcation branches for the sigmoid function $\Phi (x) = 1/(1+\exp(-15x))$. All of them show a sharp discontinuity at different noise values. We restrict the discussion to the lines representing the spatial maximum. We first focus on the full line (l2r) and the dashed line (r2l) that connect two bifurcation branches at different noise values corresponding to a hexagonal-like pattern similar to Fig.~\ref{fig:stat-patterns}(a). This clearly indicates that there is a discontinuous phase transition near the noise value indicated by the arrow. The fact that the l2r and r2l curves do not coincide further indicates that there is a hysteresis phenomenon. This conclusion is supported by the fact that the blue dot, the minimum noise value for linear stability \eqref{eq:stabilitycondition}, is to the left of both branches. This allows the possibility of  branches of dynamically unstable steady states bending backwards in noise at the phase transition point. Unstable branches are not computable with our numerical approach. Finally, we find a second bifurcation branch given by the dotted line (l2r-s) in Fig.~\ref{fig:bif-zoom}(c) corresponding to a stripe-like pattern similar to Fig.~\ref{fig:stat-patterns}(b). This branch was found by imposing a particular symmetry on the initial data, i.e., enforcing a horizontal band with activity level one.

\begin{figure*}[ht]
\centering
\subfigure
{
\includegraphics[trim={0cm 0.3cm 0.cm 0.6cm},clip, width=0.45\textwidth]{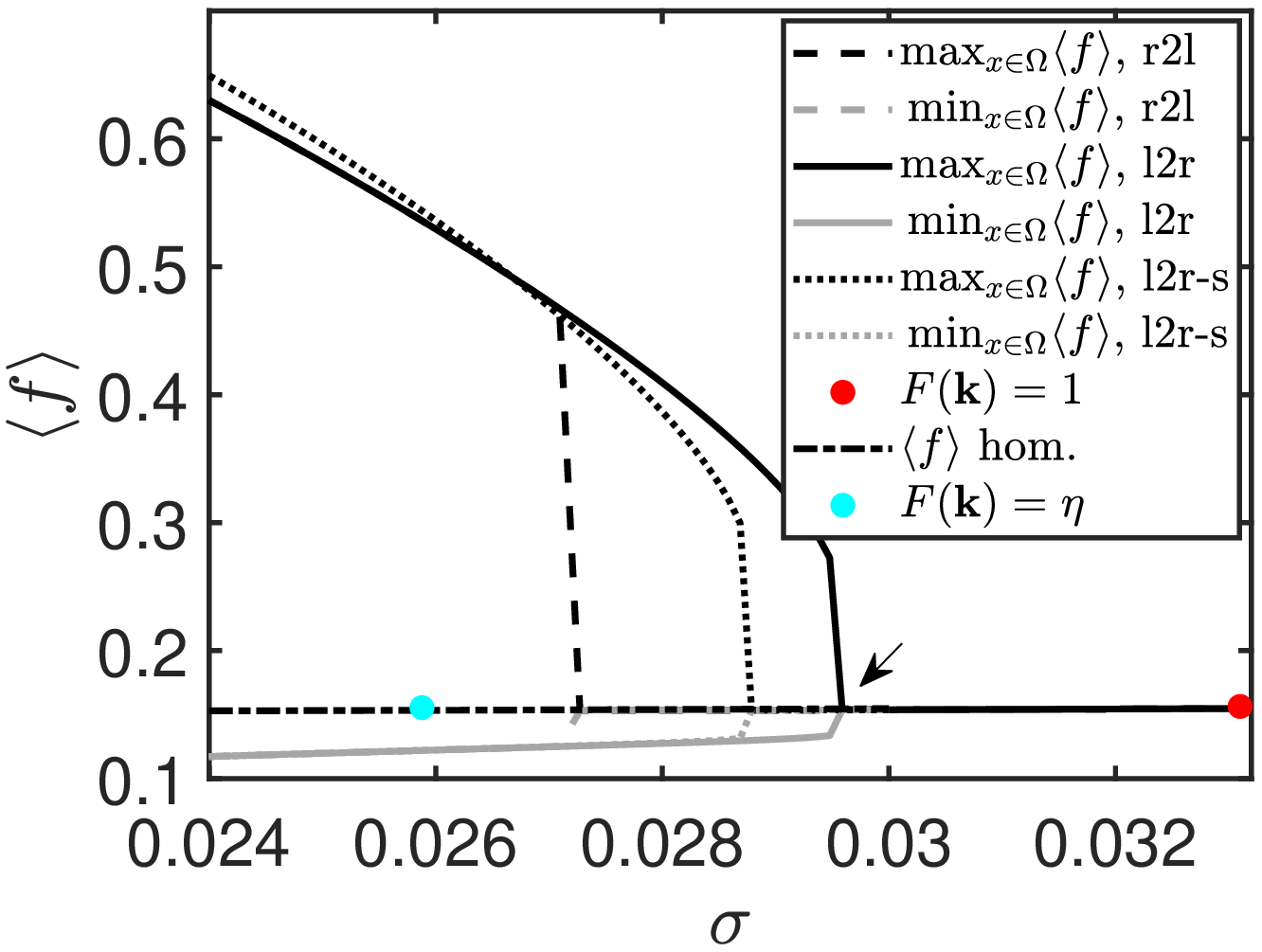}} 
\subfigure
{
\includegraphics[trim={0.0cm -0.6cm -0.5cm -0.6cm},clip,width=0.43\textwidth]{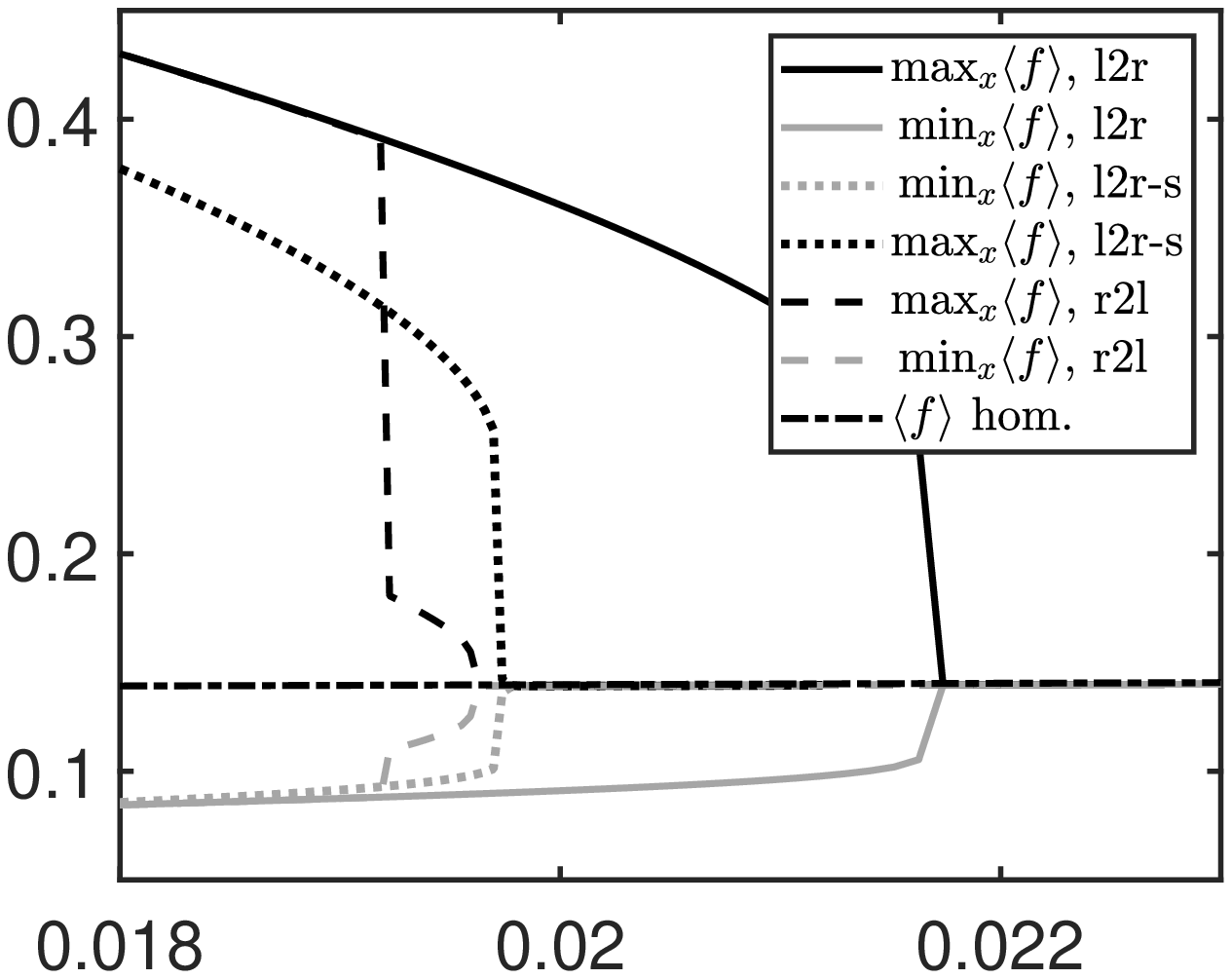}} \\
\subfigure
{
\includegraphics[trim={0cm 0.2cm 0cm 0.5cm},clip,width=0.43\textwidth]{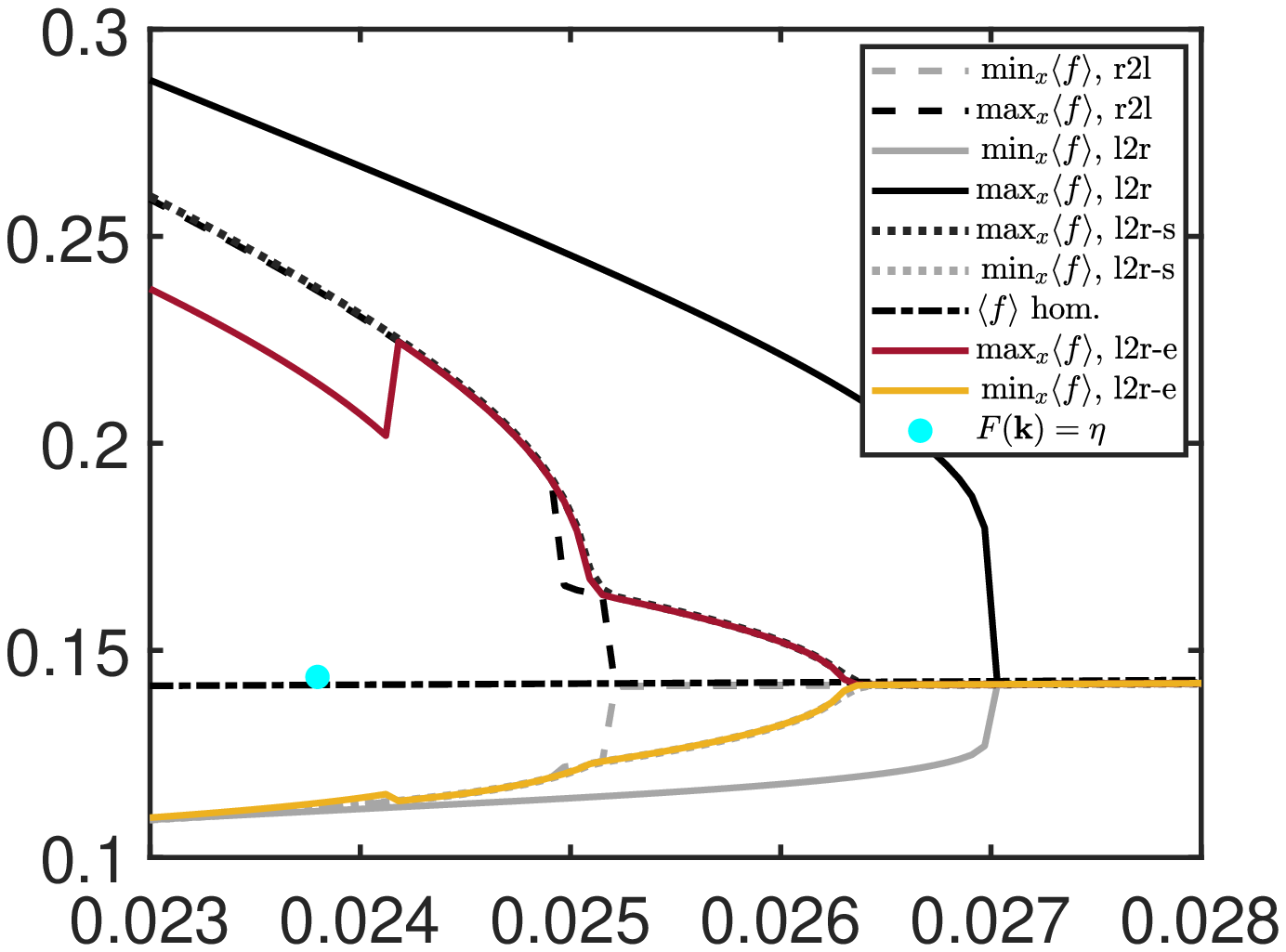}} 
\subfigure
{
\includegraphics[trim={0cm 0.2cm 0cm 0.5cm},clip,width=0.43\textwidth]{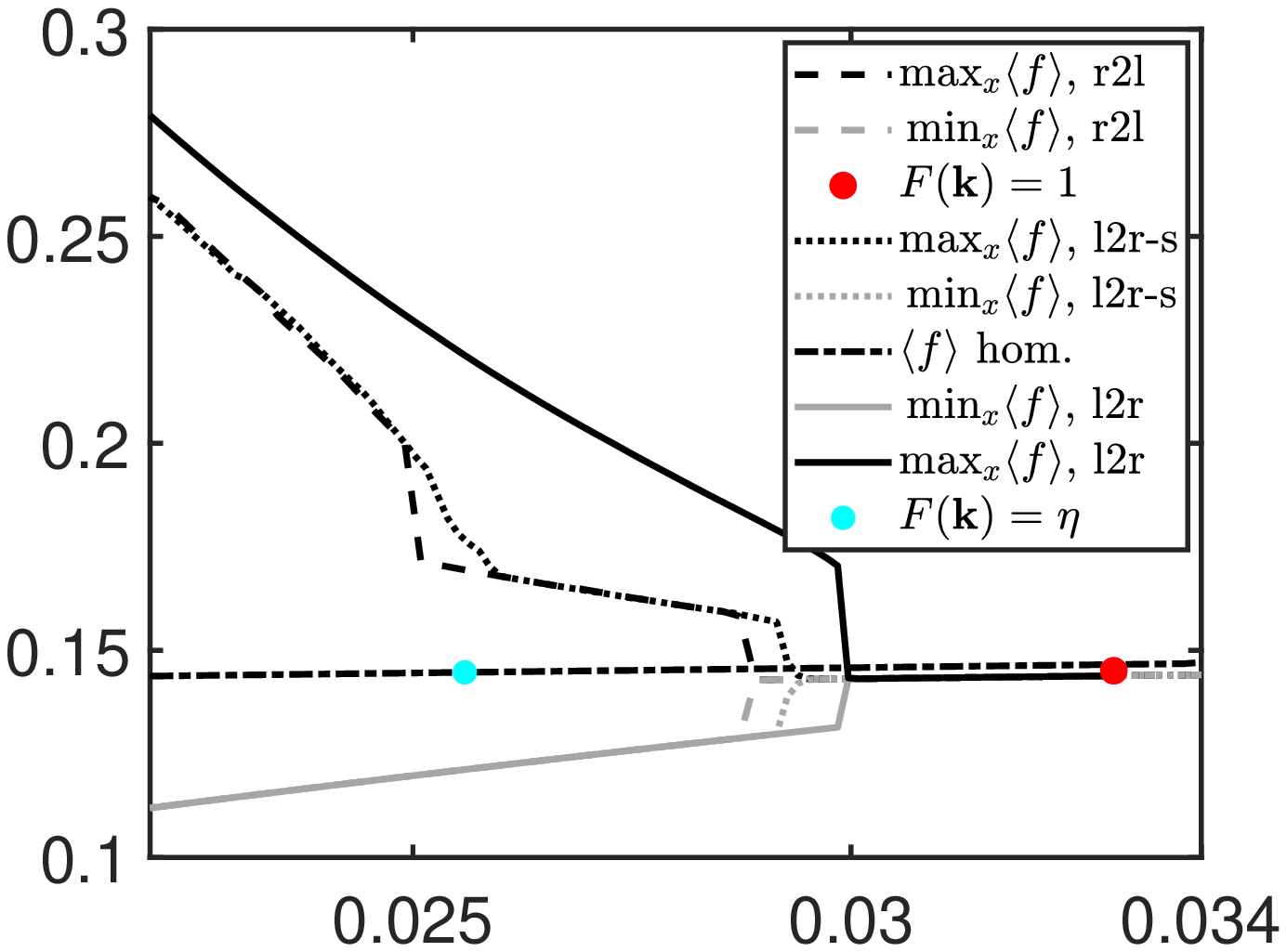}}
\caption{Bifurcation plots of $\langle f \rangle$ with respect to $\sigma$, where $W(|\bx|)=-0.005\cdot 128^2\left(1+\tanh(10-50|\bx|)\right)$ for different modulation functions $\Phi$. The red dots show  the stability threshold $F(k)=1$ for no noise while the blue dots correspond to the stability threshold \eqref{eq:stabilitycondition} with noise, where $\eta = \frac{\sigma_c}{M_\infty(\sigma_c)}$, and $\sigma_c$ is the threshold value for linear stability. Top row: $\Phi (x) = \frac{1}{1+\exp(-15x)}$ (left)=(a) , $\Phi=\Phi_\eps, \eps=0.1$ (right)=(b), bottom row: $\Phi=\Phi_\eps, \eps=0.01$ (left)=(c), and $\Phi(x) = (x)^+$  (right)=(d).
}
\label{fig:bif-zoom}
\end{figure*}

In Fig.~\ref{fig:bif-zoom}(b),(c),(d), we show analogous computations for the case of the modulation function given by $\Phi(x) = (x)^+$ and its regularisations $\Phi(x) = \Phi_\eps(x)$ with $\eps =0.1$ and $\eps =0.01$. Similarly to Fig.~\ref{fig:bif-zoom}(a), we observe a discontinuous phase transition for the full line (l2r) and the dashed line (r2l), and the linear stability blue dot is also to the left of the phase transition point as above. We remark that the blue and the red dots may lie outside the noise intervals in Fig.~\ref{fig:bif-zoom}(b),(c),(d), but they follow the same order. Similar conclusions as above lead to hysteresis phenomena and the possible existence of unstable branches not obtainable with our present numerical approach.

The case of $\eps =0.1$ in Fig.~\ref{fig:bif-zoom}(b) resembles the behaviour observed for the sigmoid function in Fig.~\ref{fig:bif-zoom}(a). The hexagonal-like patterns are the preferred stable configurations both for generic initial data, full line (l2r), and starting with small perturbations of the homogeneous stationary state, dashed line (r2l). Again stripe-like patterns are obtained by choosing specific initial data. Similar branches and the numerical observation that the hexagonal-like pattern is the most stable configuration has already been reported for a neural field model without noise \cite{VCF15}.

This behaviour changes in Fig.~\ref{fig:bif-zoom}(c),(d). The hexagonal-like patterns are still the preferred stable configurations for generic initial data, full line (l2r). However, starting with small perturbations of the homogeneous stationary state, dashed line (r2l), we connect to the stripe-like bifurcation branch, dotted line (l2r).

The bifurcation branches and their dynamics gets richer as the regularisation parameter gets smaller. We observe that for $\eps =0.01$ in Fig.~\ref{fig:bif-zoom}(c), there is an additional branch, dark red line (l2r-e), leading to eye-like patterns as in Fig.~\ref{fig:stat-patterns}(c). This branch jumps to the stripe-like pattern for larger noise values. It is difficult to extract information on the range of noise values $\sigma\in [0.025,0.0265]$ since the branch, dark red line (l2r), does not show a sharp transition point while the dashed line (r2l) does. However, this becomes much clearer in the limiting case of the positive part in Fig.~\ref{fig:bif-zoom}(d). We observe two sharper discontinuous transition points in the stripe-like branch, leading to an intermediate pure-stripe branch, $\sigma\in [0.025,0.028]$, before jumping to the homogeneous state, see the dashed line (r2l) and the dotted line (l2r).

\begin{figure}[ht]
\centering
\subfigure
{
\includegraphics[trim={0.05cm 0.cm 1.2cm 0.4cm},clip, width=0.252\textwidth]{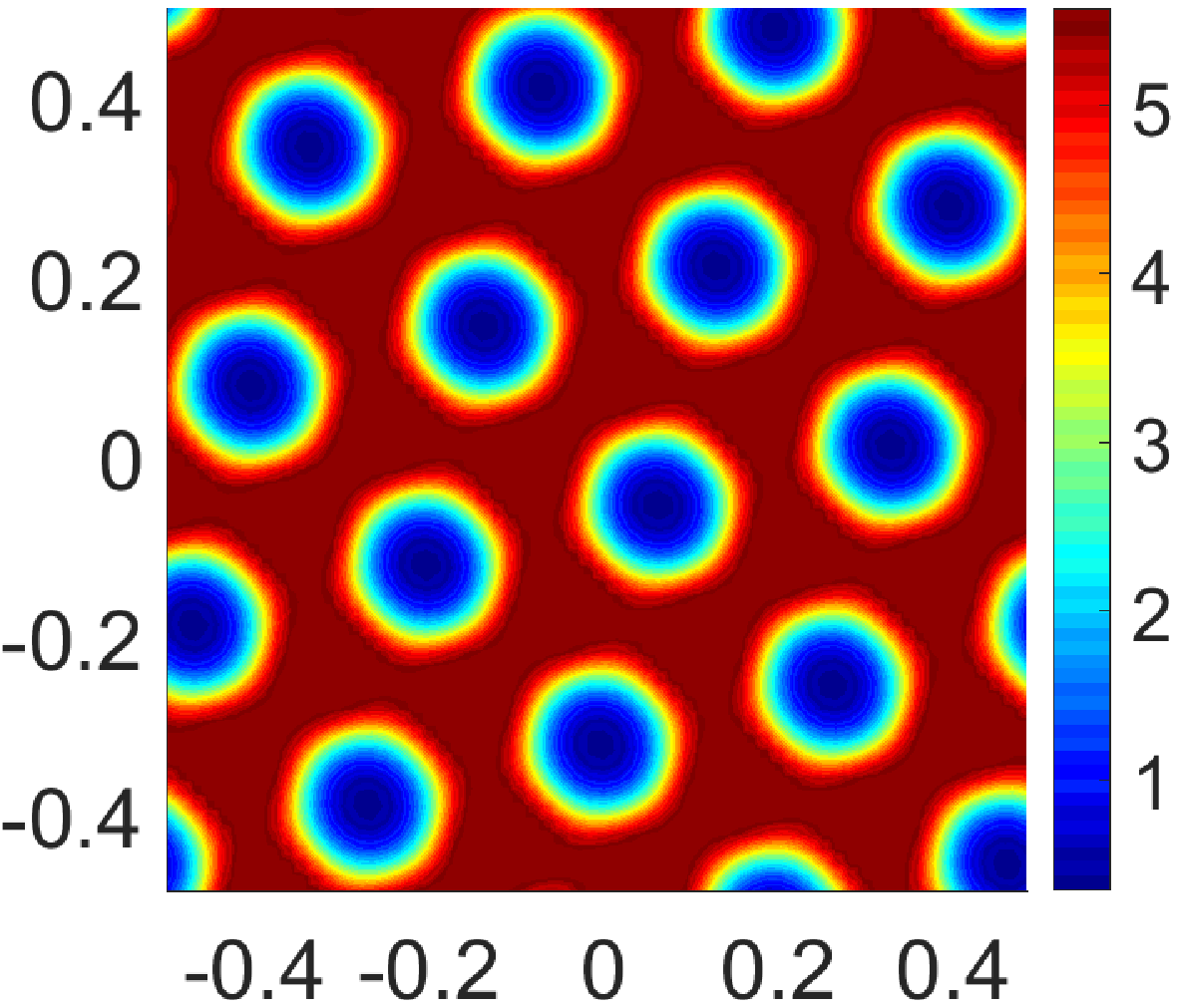}}
\subfigure
{
\includegraphics[trim={0.05cm 0.cm 1.2cm 0.4cm},clip, width=0.254\textwidth]{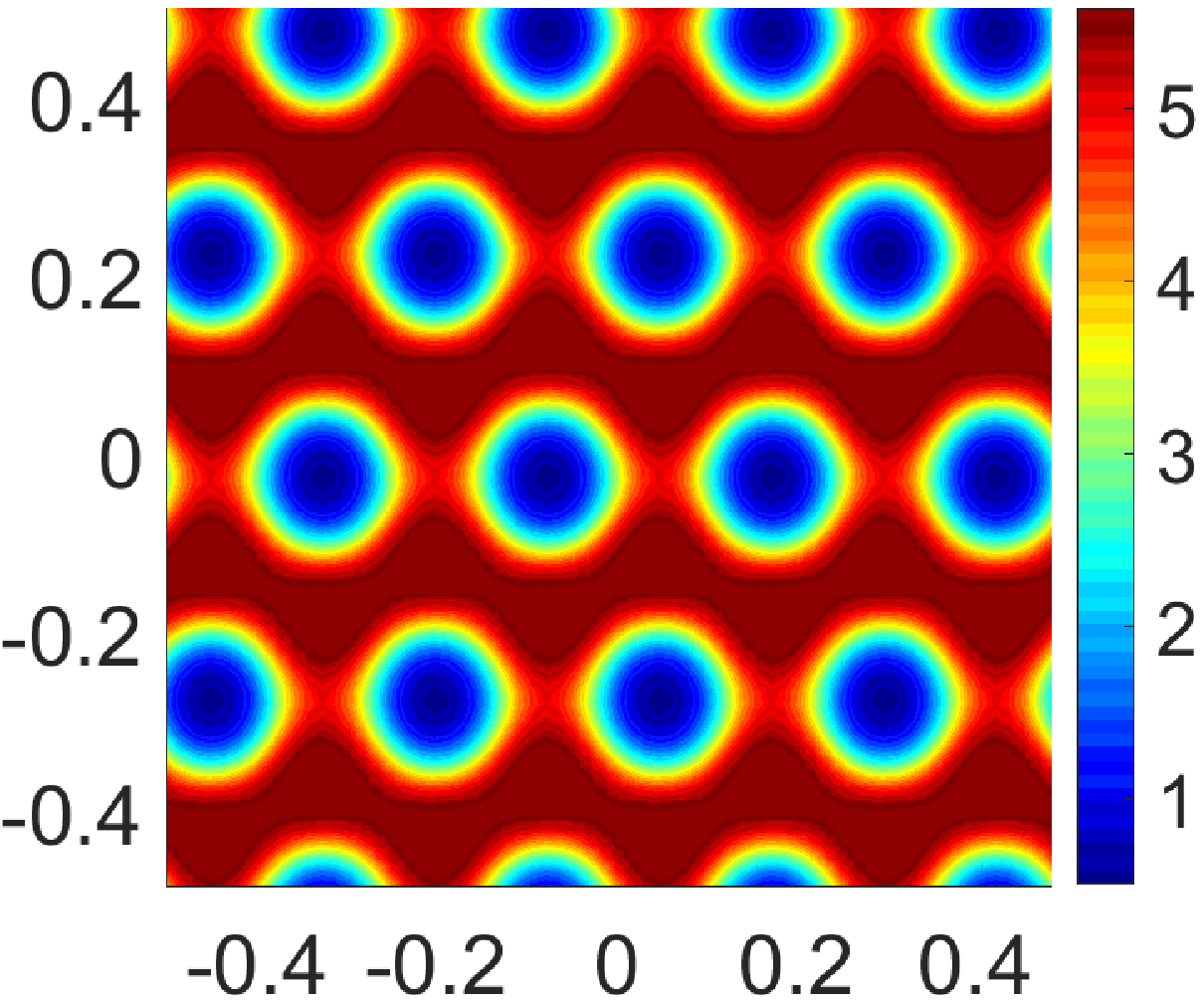}}
\subfigure
{
\includegraphics[trim={0.05cm 0.cm 1.2cm 0.4cm},clip,width=0.252\textwidth]{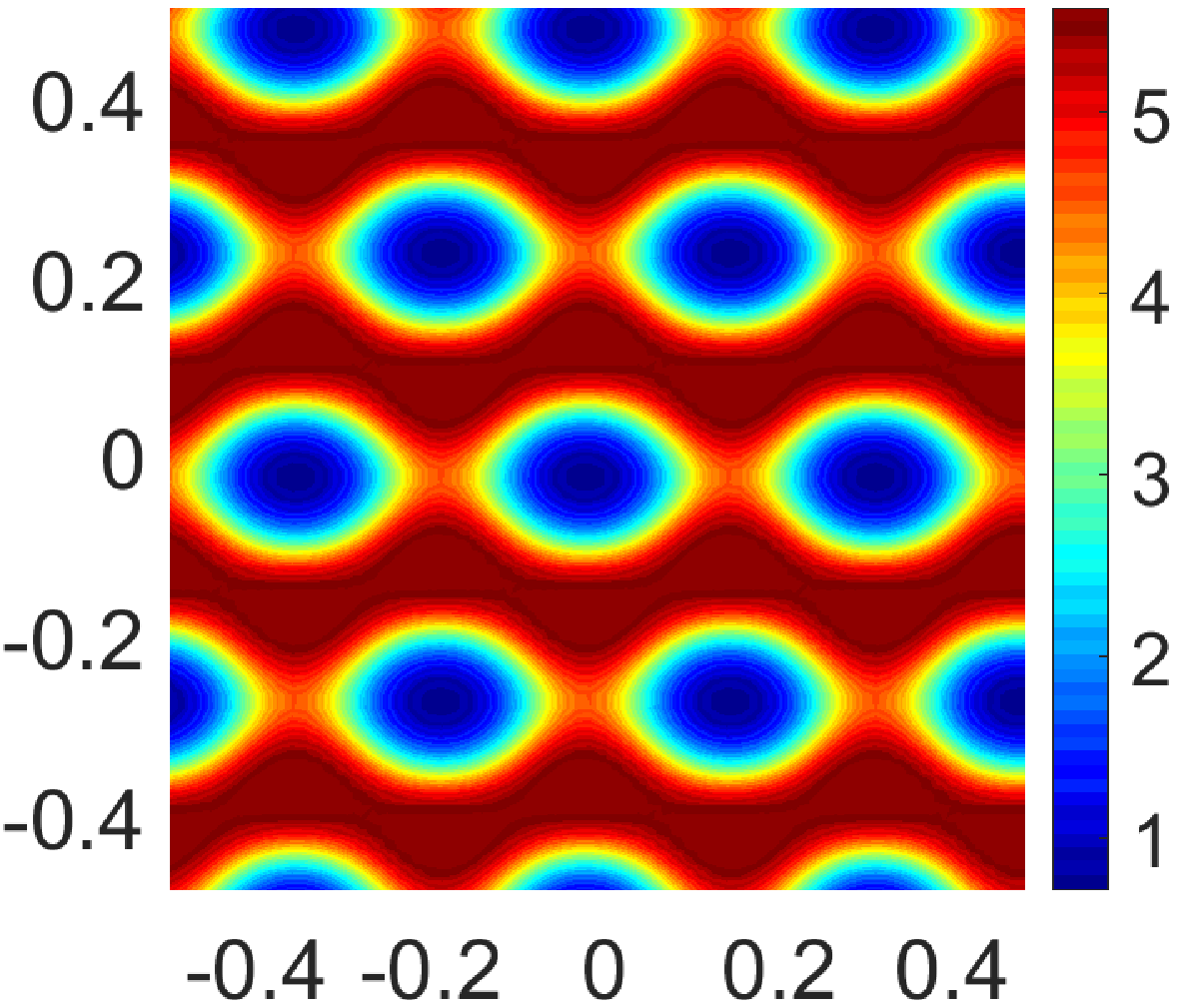}}
\caption{Stationary patterns of $f^\beta(s=0)$ at $\sigma=0.022$ for $\Phi(x) = \Phi_\eps(x), \eps =0.01$ with $W(|\bx|)=-0.005\cdot 128^2\left(1+\tanh(10-50|\bx|)\right)$: hexagonal-like (a), stripe-like (b), eye-like (c). Left to right: black, dotted, and red line in Fig.~\ref{fig:bif-zoom}(c).}
\label{fig:stat-patterns}
\end{figure}

\section{Concluding remarks}

The first conclusion of our analysis is that the mean-field limit of the grid cell model \eqref{eq:motherODE} with constant external input introduced by \cite{burakfiete,coueyetal,BurakFieteNoise}, presents phase transitions driven by the noise strength as demonstrated in Figs.~\ref{fig:bif-zoom} and \ref{fig:stat-patterns}.
This behaviour resembles the phenomena appearing in the classical Kuramoto model \cite{Kura,Kura2,RevModPhys,kuramoto-transitions,CGPS20} for synchronisation and other neural field models \cite{TouboulSIAM} in the computational neuroscience literature. 

It is shown that the homogeneous in space stationary state is linearly unstable for small noise strength, similarly to basic ring and neural field models \cite{MB,KE13,BAC19}. We numerically analysed the bifurcation diagram of stationary patterns showing the appearance of different branches identified by their symmetries, see Figs.~\ref{fig:bif-zoom} and \ref{fig:stat-patterns}. Our numerical experiments with random initial data demonstrate that the stationary hexagonal-like pattern in space of the activity level of neurons in Fig.~\ref{fig:stat-patterns}(a), leading to the solid black bifurcation branches in Fig.~\ref{fig:bif-zoom}, has the largest basin of attraction.  Moreover, the numerical simulations indicate that there is a sharp transition in the mean activity level together with a hysteresis phenomenon suggesting a discontinuous phase transition. Whether more stationary network patterns exist is another interesting topic.

The crucial implication of this phase transition on the rats' navigation path is that the larger the noise the less localised are the spatial firing fields of each neuron. This can be observed in the bottom row of Fig.~\ref{fig:single} which shows that the firing field (coloured in red) gets denser as the noise increases. Moreover, there is a sharp value of the noise after which there is no localisation at all, leading the rats to not being able to orientate themselves in physical space. In other words, the point of transition from a homogeneous pattern (all neurons have the same mean activity level) to a non-homogeneous pattern (neurons at different locations in the network have different mean activity levels) gives an upper bound for the noise strength for which single grid cells no longer can fire in a hexagonal pattern in physical space when connected with the rats movements through \eqref{eq:ext}.

With a hexagonal network activity configuration as in Fig.~\ref{fig:stat-patterns}(a), a single neuron can create hexagonal neural field patterns in physical space as in the third row of Fig.~\ref{fig:single}. Exactly how the firing fields in physical space of a single neuron are affected by initial network activity patterns as the ones in Fig.~\ref{fig:stat-patterns}(b),(c) remains to be investigated. 

From the methodological viewpoint, we remark that as the bifurcation branches are computed using a numerical approximation of the PDE \eqref{eq:PDE}--\eqref{eq:PDEbc}, they can differ slightly from the actual branches of the PDE itself. To study bifurcations and phase transitions of the nonlinear PDE analytically will require sophisticated mathematical tools, and they will be investigated elsewhere.

From the computational neuroscience viewpoint, we expect that noise driven phase transitions will also naturally appear in related attractor dynamic models as the ones in \cite{BurakFieteNoise,AB20}. Instability of homogeneous stationary network patterns should also play an important role therein. Additional investigations of more realistic models of coupled place and grid cells are needed. This will allow to connect with experiments and further contribute to the challenge of how noise affects network dynamics in  \cite[Future Issue 3]{tenyears}.

\section*{Acknowledgements}{The authors are grateful to Edvard I. Moser for helpful discussions and advice. 
JAC was supported by the Advanced Grant Nonlocal-CPD (Nonlocal PDEs for Complex Particle Dynamics: Phase Transitions, Patterns and Synchronization) of the European Research Council Executive Agency (ERC) under the European Union's Horizon 2020 research and innovation programme (grant agreement No. 883363).
HH was supported in part by the project Waves and Nonlinear Phenomena (WaNP), no 90104100, from the Research Council of Norway. Parts of this work was conducted while SS was employed at the Department of Mathematical Sciences, NTNU.}

\appendix

\section{Numerical approach}
\label{app:numerics}
This section contains details on the numerical approach. We re-write the system \eqref{eq:PDE} as 
\begin{align*}
\tau \frac{\partial f^\beta}{\partial t} =\frac{\partial}{\partial s}\Bigg( f^\beta \frac{\partial}{\partial s} \Bigg[\frac{1}{2}\left(\Phi\big(W,f,B^\beta\big) -s\right)^2 + \sigma \log f^\beta \Bigg]
\Bigg),
\end{align*}
and discretise the system using the first-order (in time $t$ and in $s$) version of the finite volume numerical method in \cite{CCH} at every $\bx\in\Omega$ on the spatial mesh.

In all figures in this paper, the time scale is set to $\tau = 10$ms, and the computations are run on a $64 \times 64 \times 64$ $(\bx,s)$-grid for each neuron type $\beta$. The activity $s$ is calculated on the interval $[0,1.3]$ and $\bx \in [-0.5,0.5]^2$. We impose no-flux boundary conditions in $s$ and periodic in $\bx$. We carefully checked that the support of the distribution in $s$ is essentially inside the interval $[0,1.3]$ for all times such that imposing no-flux boundary condition on the right end is a good approximation for $s\in [0,\infty)$.
The initial distribution is prescribed by randomly setting the activity at one percent of the grid points $\{\bx_n\}_n$, denoted $\{\bx_{n_i}\}_i$, in the four sheets to 1, i.e.,
\begin{align*}
f^{\beta}_0(\bx_n,s) = \sum_{i} \delta_{\bx_{n_i},1}(\bx,s) + \sum_{n \notin \{n_i\}_i }\delta_{\bx_n, 0}(s), 
\end{align*}
with $\beta =1,\dots,4$, where the positions $\bx_{n_i}$ are chosen randomly. The spatial preference is set by shifting the connectivity matrix determined by $W$ one cell to the north, south, west, or east, respectively.  

All figures in Section \ref{sec:stability}  are computed with $B=3$. The value of each bifurcation branch in Fig.~\ref{fig:bif-zoom} is, for each $\sigma$, found by a continuation method in $\sigma$ running the simulation for a minimum time of $2000$ms. The calculation of the branch value is stopped when either the numerical time derivative satisfies the numerical equivalent of 
\begin{align*}
  \frac{d}{dt}\bigg\|\sum_\beta f^\beta \bigg\|_{L^1(\Omega \times [0,\infty))} \leq 10^{-8},
\end{align*} 
or the maximum time of $6000\,$ms is reached. The respective noise intervals are divided into at least 100 points. 

\subsection{Grid refinement of the numerical method}

To briefly check the robustness of the numerical method described in the manuscript, we performed a standard grid refinement analysis within a computationally feasible computational range from $n=64$ to $n=256$ cells in $\bx$ and $s$. 
For $n=256$, the initial data $f^{\beta,0}_\Delta(x)|_{s=1}$ is randomly set to $1/256$ at 1\% of the locations $x$, and $f^{\beta,0}_\Delta(x)|_{s=0}$ is set to $1/256$ on the complement. The rest is set to zero. On the coarser grids, we have used piecewise averages of this initial data. In Fig.~\ref{fig:gridrefinement}, the solutions on a grid consisting of $n^3$ cubes, with $n=32, 64, 128,$ and $256$, are plotted. The corresponding $L^1$ and $L^2$ errors can be found in the table below. The coarser solutions ($n=32,64,128$) are compared to the one computed on the grid consisting of $256^3$ cubes. It can be observed from the table and the figure that the numerical method used is stable when refining the grid. 
\begin{figure}[ht]
\centering
\subfigure[$n=32$]{
\includegraphics[trim={0cm 1.cm 0cm 0cm},clip, width=0.23\textwidth]{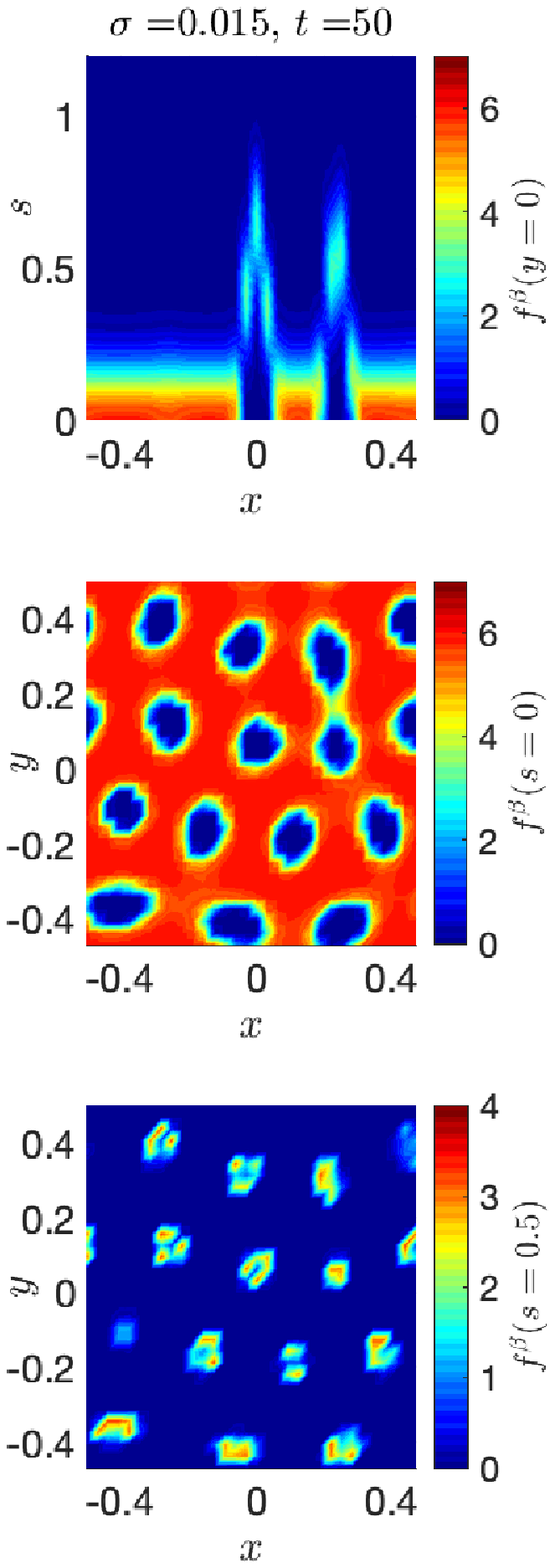}}
\subfigure[$n=64$]{
\includegraphics[trim={0cm 1.cm 0cm 0cm},clip,width=0.23\textwidth]{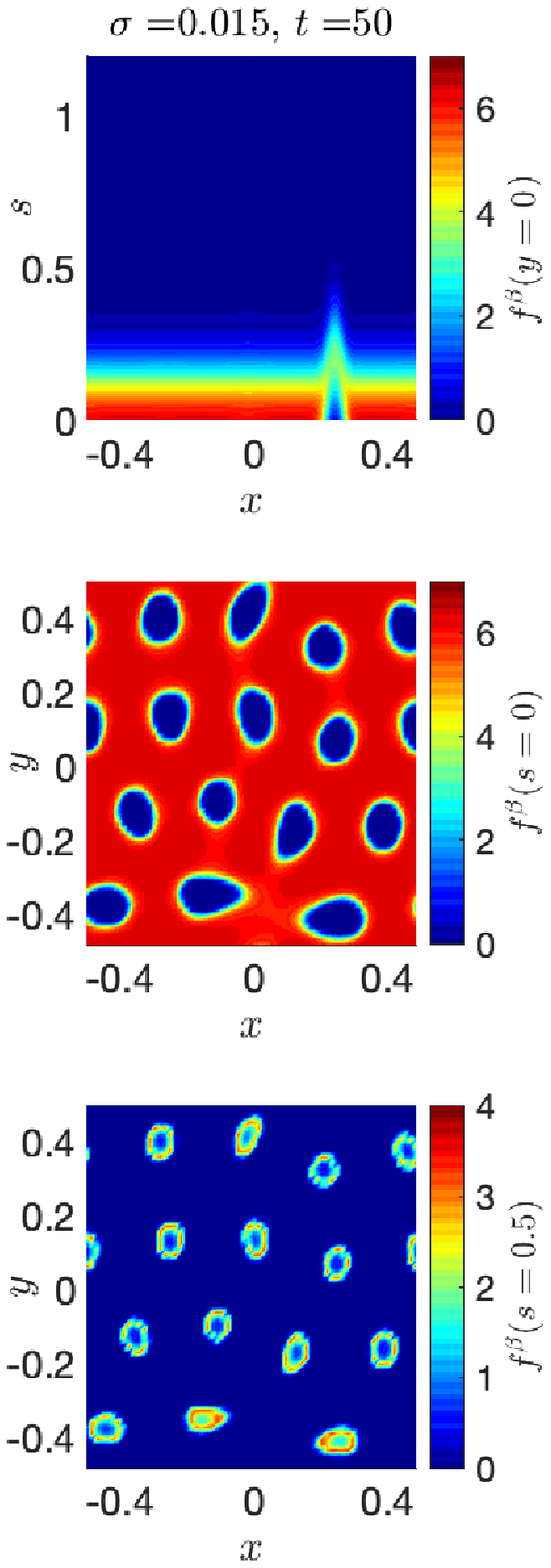}}
\subfigure[$n=128$]{
\includegraphics[trim={0cm 1.cm 0cm 0cm},clip,width=0.23\textwidth]{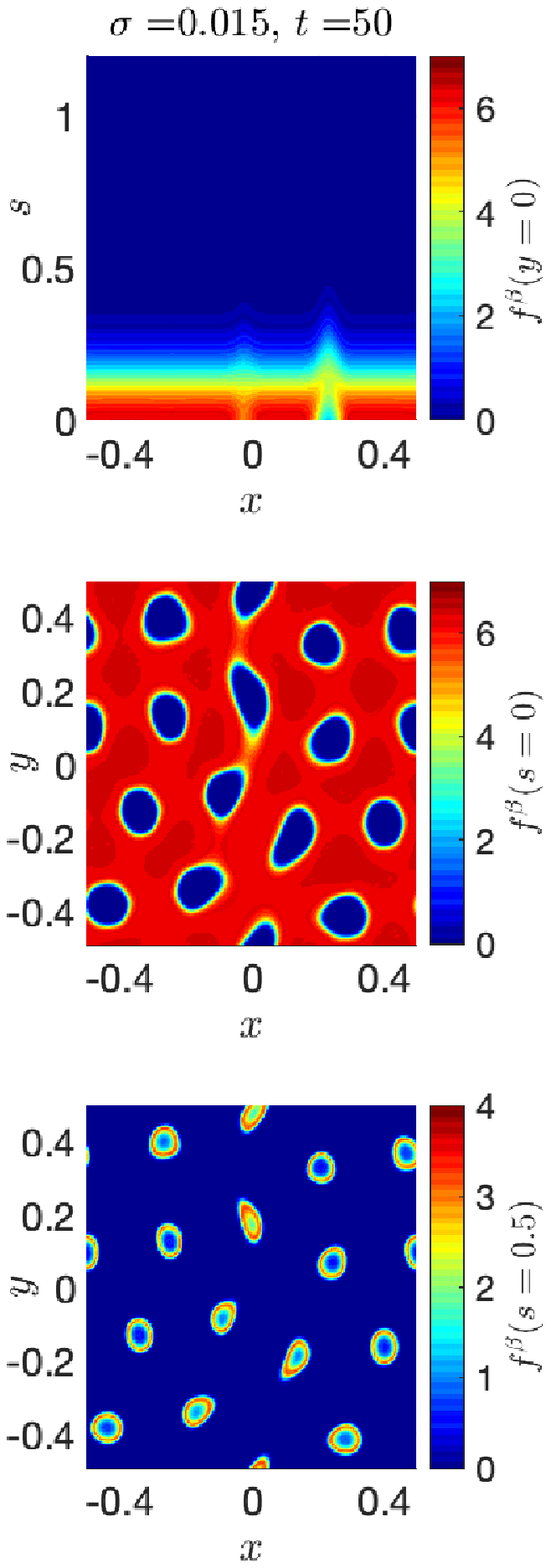}}
\subfigure[$n=256$]{\includegraphics[trim={0cm 1.cm 0cm 0cm},clip,width=0.23\textwidth]{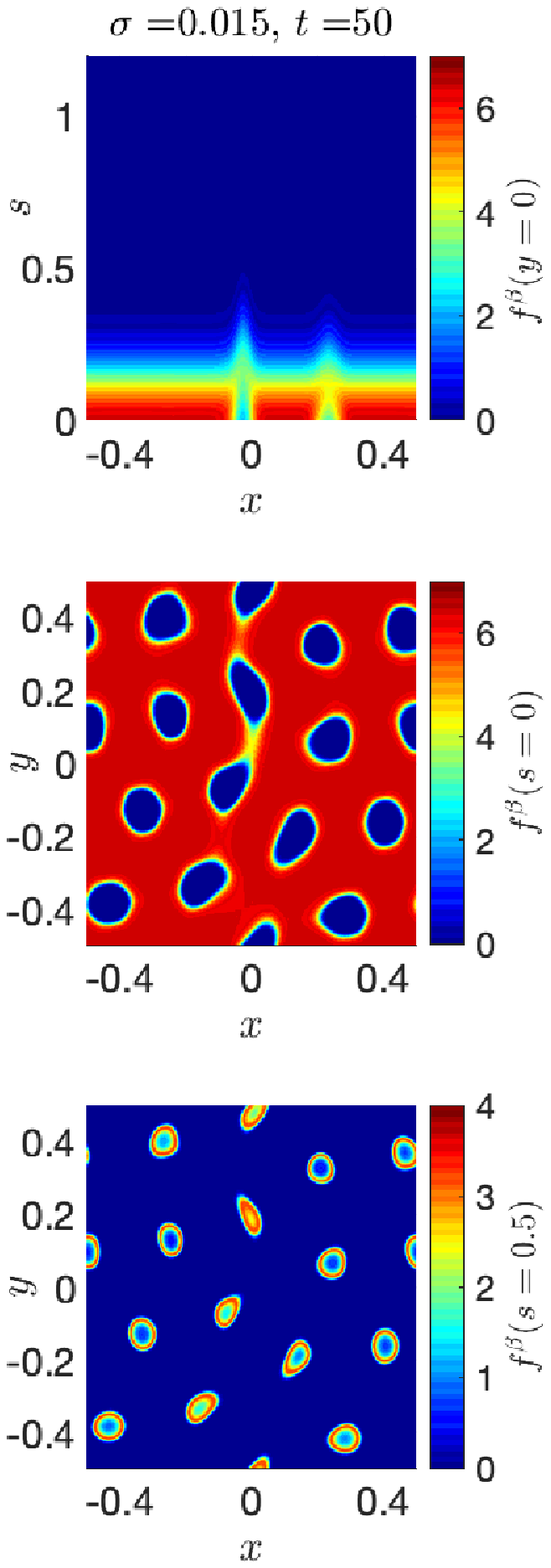}}
\caption{Grid refinement for $\sigma = 0.015$, $B=3$, $W(|\bx|)=-0.005 \cdot 128^2(1+1\tanh(50(0.2-|\bx|)))$, and $\Phi(x) = 1/(1+\exp(-15x))$, at $t=50$ms.}
\label{fig:gridrefinement}
\end{figure}

\begin{table}[ht!]
\centering
\begin{tabular}{r|l|l|c|c}
$n$ & $L^1$ & $L^2$ & OOC $L^1$ & OOC $L^2$ \\
\hline
32 & 0.1033 & 0.1993 & $-$ & $-$ \\
64 & 0.0565 & 0.1085 & 0.87 & 0.88 \\
128 & 0.0098 & 0.0175 & 2.53 & 2.64 \\
\end{tabular}
\caption{$L^1$ and $L^2$ (in $\bx$) errors and order of convergence (OOC) of the mean for $n=32, 64, 128$.}
\end{table}


\bibliographystyle{habbrv}
\bibliography{references}

\begin{thebibliography}{10}

\bibitem{RevModPhys}
J.~A. Acebr\'on, L.~L. Bonilla, C.~J. P\'erez~Vicente, F.~Ritort, and
  R.~Spigler.
\newblock The {K}uramoto model: A simple paradigm for synchronization
  phenomena.
\newblock {\em Rev. Mod. Phys.}, 77:137--185, 2005.

\bibitem{AB20}
H.~Agamon and Y.~Burak.
\newblock A theory of joint attractor dynamics in the hippocampus and the
  entorhinal cortex accounts for artificial remapping and grid cell
  field-to-field variability.
\newblock {\em eLife}, 9:e56894, 2020.

\bibitem{Amari1977}
S.-i. Amari.
\newblock Dynamics of pattern formation in lateral-inhibition type neural
  fields.
\newblock {\em Biol. Cybernet.}, 27(2):77--87, 1977.

\bibitem{AG95}
G.~B. Arous and A.~Guionnet.
\newblock Large deviations for {L}angevin spin glass dynamics.
\newblock {\em Probab. Theory Related Fields}, 102(4):455--509, 1995.

\bibitem{bressloff2012}
P.~C. Bressloff.
\newblock Spatiotemporal dynamics of continuum neural fields.
\newblock {\em J. Phys. A}, 45(3):033001, 109, 2012.

\bibitem{B19}
P.~C. Bressloff.
\newblock Stochastic neural field model of stimulus-dependent variability in
  cortical neurons.
\newblock {\em PLOS Computational Biology}, 15(3):1--33, 2019.

\bibitem{burakfiete}
Y.~Burak and I.~Fiete.
\newblock Accurate path integration in continuous attractor network models of
  grid cells.
\newblock {\em PLoS Comput. Biol.}, 5(2):e1000291, 2009.

\bibitem{BurakFieteNoise}
Y.~Burak and I.~Fiete.
\newblock Fundamental limits on persistent activity in networks of noisy
  neurons.
\newblock {\em PNAS}, 109:17645--17650, 2012.

\bibitem{BAC19}
{\'A}.~Byrne, D.~Avitabile, and S.~Coombes.
\newblock Next-generation neural field model: The evolution of synchrony within
  patterns and waves.
\newblock {\em Physical Review E}, 99(1):012313, 2019.

\bibitem{CCR11}
J.~A. Ca\~{n}izo, J.~A. Carrillo, and J.~Rosado.
\newblock A well-posedness theory in measures for some kinetic models of
  collective motion.
\newblock {\em Math. Models Methods Appl. Sci.}, 21(3):515--539, 2011.

\bibitem{CT18}
T.~Cabana and J.~D. Touboul.
\newblock Large deviations for randomly connected neural networks: {I}.
  {S}patially extended systems.
\newblock {\em Adv. in Appl. Probab.}, 50(3):944--982, 2018.

\bibitem{CTRM06}
D.~Cai, L.~Tao, A.~V. Rangan, and D.~W. McLaughlin.
\newblock Kinetic theory for neuronal network dynamics.
\newblock {\em Commun. Math. Sci.}, 4(1):97--127, 2006.

\bibitem{CCH}
J.~A. Carrillo, A.~Chertock, and Y.~Huang.
\newblock A finite-volume method for nonlinear nonlocal equations with a
  gradient flow structure.
\newblock {\em Commun. Comput. Phys.}, 17(1):233--258, 2015.

\bibitem{kuramoto-transitions}
J.~A. Carrillo, Y.-P. Choi, and L.~Pareschi.
\newblock Structure preserving schemes for the continuum {K}uramoto model:
  phase transitions.
\newblock {\em J. Comput. Phys.}, 376:365--389, 2019.

\bibitem{CCS21}
J.~A. Carrillo, A.~Clini, and S.~Solem.
\newblock The mean field limit of stochastic differential equation systems
  modelling grid cells, 2021, arXiv:2112.06213.

\bibitem{CCM11}
J.~A. Carrillo, S.~Cordier, and S.~Mancini.
\newblock A decision-making {F}okker--{P}lanck model in computational
  neuroscience.
\newblock {\em J. Math. Biol.}, 63(5):801--830, 2011.

\bibitem{CCM13}
J.~A. Carrillo, S.~Cordier, and S.~Mancini.
\newblock One dimensional {F}okker--{P}lanck reduced dynamics of decision
  making models in computational neuroscience.
\newblock {\em Commun. Math. Sci.}, 11(2):523--540, 2013.

\bibitem{CGPS20}
J.~A. Carrillo, R.~S. Gvalani, G.~A. Pavliotis, and A.~Schlichting.
\newblock Long-time behaviour and phase transitions for the
  {M}c{K}ean--{V}lasov equation on the torus.
\newblock {\em Arch. Ration. Mech. Anal.}, 235(1):635--690, 2020.

\bibitem{coueyetal}
J.~J. Couey, A.~Witoelar, S.-J. Zhang, K.~Zheng, J.~Ye, B.~Dunn, R.~Czajkowski,
  M.-B. Moser, E.~I. Moser, Y.~Roudi, and M.~P. Witter.
\newblock Recurrent inhibitory circuitry as a mechanism for grid formation.
\newblock {\em Nat. Neurosci.}, 16:318--324, 2013.

\bibitem{EC1979}
G.~B. Ermentrout and J.~D. Cowan.
\newblock A mathematical theory of visual hallucination patterns.
\newblock {\em Biol. Cybernet.}, 34(3):137--150, 1979.

\bibitem{Ermentrout2010}
G.~B. Ermentrout and D.~H. Terman.
\newblock {\em Mathematical {F}oundations of {N}euroscience}, volume~35 of {\em
  Interdisciplinary Applied Mathematics}.
\newblock Springer, New York, 2010.

\bibitem{FI15}
O.~Faugeras and J.~Inglis.
\newblock Stochastic neural field equations: a rigorous footing.
\newblock {\em J. Math. Biol.}, 71(2):259--300, 2015.

\bibitem{FTC09}
O.~Faugeras, J.~Touboul, and B.~Cessac.
\newblock A constructive mean-field analysis of multi population neural
  networks with random synaptic weights and stochastic inputs.
\newblock {\em Frontiers in Computational Neuroscience}, 3, 2009.

\bibitem{Gardner2022}
R.~J. Gardner, E.~Hermansen, M.~Pachitariu, Y.~Burak, N.~A. Baas, B.~A. Dunn,
  M.-B. Moser, and E.~I. Moser.
\newblock Toroidal topology of population activity in grid cells.
\newblock {\em Nature}, 2022.

\bibitem{Go16}
F.~Golse.
\newblock On the dynamics of large particle systems in the mean field limit.
\newblock In {\em Macroscopic and large scale phenomena: coarse graining, mean
  field limits and ergodicity}, volume~3 of {\em Lect. Notes Appl. Math.
  Mech.}, pages 1--144. Springer, [Cham], 2016.

\bibitem{G97}
A.~Guionnet.
\newblock Averaged and quenched propagation of chaos for spin glass dynamics.
\newblock {\em Probab. Theory Related Fields}, 109(2):183--215, 1997.

\bibitem{gridcells}
T.~Hafting, M.~Fyhn, S.~Molden, M.-B. Moser, and E.~I. Moser.
\newblock Microstructure of a spatial map in the entorhinal cortex.
\newblock {\em Nature}, 436:801--806, 2005.

\bibitem{HJ2}
M.~Hauray and P.-E. Jabin.
\newblock {$N$}-particles approximation of the {V}lasov equations with singular
  potential.
\newblock {\em Arch. Ration. Mech. Anal.}, 183(3):489--524, 2007.

\bibitem{HJ1}
M.~Hauray and P.-E. Jabin.
\newblock Particle approximation of {V}lasov equations with singular forces:
  propagation of chaos.
\newblock {\em Ann. Sci. \'{E}c. Norm. Sup\'{e}r. (4)}, 48(4):891--940, 2015.

\bibitem{Hopfield84}
J.~J. Hopfield.
\newblock Neurons with graded response have collective computational properties
  like those of two-state neurons.
\newblock {\em Proceedings of the National Academy of Sciences},
  81(10):3088--3092, 1984.

\bibitem{J14}
P.-E. Jabin.
\newblock A review of the mean field limits for {V}lasov equations.
\newblock {\em Kinet. Relat. Models}, 7(4):661--711, 2014.

\bibitem{K14}
Z.~P. Kilpatrick.
\newblock Coupling layers regularizes wave propagation in stochastic neural
  fields.
\newblock {\em Phys. Rev. E}, 89:022706, 2014.

\bibitem{KE13}
Z.~P. Kilpatrick and B.~Ermentrout.
\newblock Wandering bumps in stochastic neural fields.
\newblock {\em SIAM J. Appl. Dyn. Syst.}, 12(1):61--94, 2013.

\bibitem{KP17}
Z.~P. Kilpatrick and D.~B. Poll.
\newblock Neural field model of memory-guided search.
\newblock {\em Phys. Rev. E}, 96:062411, 2017.

\bibitem{Kura}
Y.~Kuramoto.
\newblock Rhythms and turbulence in populations of chemical oscillators.
\newblock {\em Phys. A}, 106(1-2):128--143, 1981.

\bibitem{LS84}
P.-L. Lions and A.-S. Sznitman.
\newblock Stochastic differential equations with reflecting boundary
  conditions.
\newblock {\em Comm. Pure Appl. Math.}, 37(4):511--537, 1984.

\bibitem{MB}
J.~N. MacLaurin and P.~C. Bressloff.
\newblock Wandering bumps in a stochastic neural field: a variational approach.
\newblock {\em Phys. D}, 406:132403, 9, 2020.

\bibitem{McNaughtonMoser}
B.~McNaughton, E.~Moser, and M.-B. Moser.
\newblock Spatial representation in the hippocampal formation: a history.
\newblock {\em Nat. Neurosci.}, 20:1448--1464, 2017.

\bibitem{McNaughtonetal}
B.~L. McNaughton, F.~P. Battaglia, O.~Jensen, E.~I. Moser, and M.-B. Moser.
\newblock Path integration and the neural basis of the 'cognitive map'.
\newblock {\em Nature Reviews Neuroscience}, 7(8):663--678, 2006.

\bibitem{MS02}
O.~Moynot and M.~Samuelides.
\newblock Large deviations and mean-field theory for asymmetric random
  recurrent neural networks.
\newblock {\em Probab. Theory Related Fields}, 123(1):41--75, 2002.

\bibitem{Mu72}
B.~Muckenhoupt.
\newblock Hardy's inequality with weights.
\newblock {\em Studia Math.}, 44:31--38, 1972.

\bibitem{Murray03}
J.~Murray.
\newblock {\em Mathematical {B}iology}.
\newblock Springer-Verlag New York, 2002.

\bibitem{whiskerbarrel}
D.~Pinto, J.~C. Brumberg, D.~J. Simons, and B.~Ermentrout.
\newblock A quantitative population model of whisker barrels: Re-examining the
  {W}ilson--{C}owan equations.
\newblock {\em J. Comput. Neurosci.}, 3:247--264, 1996.

\bibitem{rolls2010noisy}
E.~T. Rolls and G.~Deco.
\newblock {\em The {N}oisy {B}rain: {S}tochastic {D}ynamics as a {P}rinciple of
  {B}rain {F}unction}.
\newblock Oxford University Press, Oxford, 2010.

\bibitem{RBI17}
O.~Roustant, F.~Barthe, and B.~Iooss.
\newblock Poincar\'{e} inequalities on intervals---application to sensitivity
  analysis.
\newblock {\em Electron. J. Stat.}, 11(2):3081--3119, 2017.

\bibitem{tenyears}
D.~C. Rowland, Y.~Roudi, M.-B. Moser, and E.~I. Moser.
\newblock Ten years of grid cells.
\newblock {\em Annu. Rev. Neurosci.}, 39:19--40, 2016.

\bibitem{Kura2}
H.~Sakaguchi, S.~Shinomoto, and Y.~Kuramoto.
\newblock {Phase transitions and their bifurcation analysis in a large
  population of active rotators with mean-field coupling}.
\newblock {\em Progress of Theoretical Physics}, 79(3):600--607, 1988.

\bibitem{SA20}
H.~Schmidt and D.~Avitabile.
\newblock Bumps and oscillons in networks of spiking neurons.
\newblock {\em Chaos}, 30(3):033133, 13, 2020.

\bibitem{Sznitman1984}
A.-S. Sznitman.
\newblock Nonlinear reflecting diffusion process, and the propagation of chaos
  and fluctuations associated.
\newblock {\em J. Funct. Anal.}, 56(3):311--336, 1984.

\bibitem{Sznitman1991}
A.-S. Sznitman.
\newblock Topics in propagation of chaos.
\newblock In {\em \'{E}cole d'\'{E}t\'{e} de {P}robabilit\'{e}s de
  {S}aint-{F}lour {XIX}---1989}, volume 1464 of {\em Lecture Notes in Math.},
  pages 165--251. Springer, Berlin, 1991.

\bibitem{TouboulPhysD}
J.~Touboul.
\newblock Mean-field equations for stochastic firing-rate neural fields with
  delays: Derivation and noise-induced transitions.
\newblock {\em Physica D: Nonlinear Phenomena}, 241(15):1223--1244, 2012.

\bibitem{TouboulSIAM}
J.~Touboul, G.~Hermann, and O.~Faugeras.
\newblock Noise-induced behaviors in neural mean field dynamics.
\newblock {\em SIAM J. Appl. Dyn. Syst.}, 11(1):49--81, 2012.

\bibitem{VCF15}
R.~Veltz, P.~Chossat, and O.~Faugeras.
\newblock On the effects on cortical spontaneous activity of the symmetries of
  the network of pinwheels in visual area {V}1.
\newblock {\em J. Math. Neurosci.}, 5:Art. 11, 28, 2015.

\bibitem{WC1}
H.~Wilson and J.~Cowan.
\newblock Excitatory and inhibitory interactions in localized populations of
  model neurons.
\newblock {\em Biophys J.}, 12:1--24, 1972.

\bibitem{WC2}
H.~Wilson and J.~Cowan.
\newblock A mathematical theory of the functional dynamics of cortical and
  thalamic nervous tissue.
\newblock {\em Biol. Cybern.}, 13:55--80, 1973.

\end{thebibliography}

\end{document}